\documentclass[12pt]{article}
\usepackage{graphicx} 

\usepackage{graphicx}
\usepackage{geometry}
\usepackage{amsmath}
\usepackage{amssymb}
\usepackage{esint}
\usepackage{bm}
\usepackage{array}
\usepackage{float}
\usepackage{flafter}
\usepackage{appendix}
\usepackage{caption}
\usepackage{subcaption}
\usepackage{enumitem}
\usepackage{multirow}
\let\vec\mathbf

\usepackage[hidelinks]{hyperref}
\hypersetup{colorlinks=true,
            citecolor=blue,}

\usepackage{amsthm}
\newtheorem{theorem}{Theorem}

\usepackage[style = apa]{biblatex}
\DeclareLanguageMapping{english}{english-apa}
\title{The effect of linear dispersive errors \\ on nonlinear timestepping accuracy in the  \\ $f$-plane rotating shallow water equations}

\author{Timothy C. Andrews, Jemma Shipton, Beth A. Wingate}

\begin{document}

\maketitle

\begin{abstract}
For simulations of time evolution problems, such as weather and climate models, taking the largest stable timestep is advantageous for reducing wall-clock time. A drawback of doing so is the potential reduction in nonlinear accuracy of the numerical solution --- we investigate this for the Rotating Shallow Water Equations (RSWEs) on an $f$-plane. First, we examine how linear dispersion errors can impact the nonlinear dynamics. By deriving an alternate time evolution equation for the RSWEs, the dynamics can be expressed through interactions of three linear waves in triads. Linear dispersion errors may appear in the numerical representation of the frequency of each triad, which will impact the timestepped nonlinear dynamics. A new triadic error quantifies this by composing three stability polynomials from the oscillatory Dahlquist test equation. Second, we design two new test cases to examine the effect of timestep size in a numerical model. These tests investigate how well a timestepper replicates slow nonlinear dynamics amidst fast linear oscillations. The first test case of a Gaussian height perturbation contains a nonlinear phase shift that can be missed with a large timestep. The second set of triadic test cases excite a few linear waves to instigate specific triadic interactions. Two triadic cases are examined: one with a dominant directly resonant triad and another with near-resonances that redistribute fast mode energy into rings in spectral space. Three numerical models, including LFRic from the Met Office, are examined in these test cases with different timesteppers.
\end{abstract}

\section{Introduction}
As simulations of time evolution problems increase in spatial grid resolution and complexity, for example using data assimilation, the computational cost grows. This is particularly true for weather and climate simulations. For operational weather forecasting, the wall-clock time-to-solution is restricted so that timely, meaningful results can be produced such as daily weather forecasts. One way to meet these demands is to increase the computational power applied to a simulation. However, this is incompatible with today's many-core hardware improvements and the importance of sustainable computing practices \parencite{green_comp,lannelongue2021green}. Therefore, the next generation of weather and climate models are prioritising the reduction of their wall-clock time \parencite{Chasm}. One solution to reduce wall-clock times is to use the largest possible stable timestep that also retains sufficient accuracy.  \par
A key challenge when taking large timesteps in weather and climate models is the complexity of the governing equations, which exhibit multiscale dynamics. These equations typically contain fast, oscillatory, waves, which impose limits on the allowable timestep size of explicit methods. There can also be large variations in the timescales of the dynamical features. An example is the disparate wave speeds of Rossby, inertia-gravity, and acoustic waves in the compressible Euler equations that govern geophysical flows \parencite{Durran}. This multiscale complexity is retained in the simpler Rotating Shallow Water Equations (RSWEs), which makes these well-suited for analysing numerical discretisations \parencite{James_Kent_LFRic,Aechtner_wavelet,Ullrich_Jablonowski,Ullrich_Lauritzen} --- we consider the $f$-plane version of the RSWEs in this paper. \par
The balance between timestep size, stability and accuracy has inspired the investigation of many varieties of timestepping schemes for weather and climate models. These can be classified as explicit, implicit, and semi-implicit, and can be applied in different formulations of the equations, such as the Eulerian or semi-Lagrangian frames of reference. Each approach has its strengths and weaknesses when applied with a large timestep. Explicit timesteppers can be computationally cheap, because they reuse previously computed derivative information, but have strict stability limits from the linear oscillations and dissipation. Implicit methods may have larger stability regions, but can be limited to small timesteps to have sufficient accuracy. The lack of an obvious timestepping choice for enabling large, accurate, timesteps, is reflected by the range of methods used in dynamical cores across the world \parencite{dynamical_core_comparison}. \par
In this paper, we focus on the quantification of nonlinear timestepping error in numerical solutions to the $f$-plane RSWEs. This work is split into two distinct parts:
\begin{enumerate}[label=\Roman*]
    \item The first part (Sections \ref{section:RSWEs}, \ref{section:triad_interactions}, \ref{section:triad_errors}) will use the Fourier method to reformulate the RSWEs into a time-dependent ODE for spectral coefficients; this allows for the identification of nonlinear interactions. This ODE contains quadratically nonlinear terms, a coupling coefficient, and a triadic propagator term of $\mathcal{T} = \exp(i \Omega t)$ that depends on the triadic frequency, $\Omega$. In a novel use of linear dispersion theory, we quantify errors in $\Omega$ by composing stability polynomials from the oscillatory Dahlquist equation.
    \item The second part (Sections \ref{section:numerical_models}, \ref{section:gaussian_test_case}, \ref{section:triad_test_case}) will consider timestepped solution errors in two new test cases for the biperiodic $f$-plane RSWEs. These test cases exhibit slowly developing nonlinear dynamics alongside more rapid linear oscillations. We investigate how well timesteppers resolve these nonlinear features when used with a large timestep.
\end{enumerate}

In part I we will make use of linear dispersion analysis to construct a measure of timestepping error in nonlinear dynamics. From previous analyses of numerical schemes in simple model problems, such as Dahlquist-type equations \parencite{dahlquist1963special}, much has been understood (see \textcite{Durran}) about their behaviour; this is even though they are applied to the more complex and fully nonlinear equations that govern weather and climate dynamics. In a similar manner, we will investigate an aspect of nonlinear timestepping error by applying the Dahlquist test equation to the frequencies of the nonlinear interactions. Due to the bilinear form of the RSWEs nonlinearity, these nonlinear interactions appear at leading-order through combinations of three linear waves in \textit{triads}. \par
Triads are a well-examined phenomenon in fluid systems, e.g \parencite{craik1971non,bretherton1964resonant,Newell_rossby}, and have been studied experimentally \parencite{mcewan1971degeneration,bourget2013experimental}. Triad interactions have been shown to move a considerable amount of energy between wavenumbers and modes within numerical fluid models \parencite{Ward_Dewar, smith1999transfer} and are also posited to be responsible for generating longer timescale dynamics on a planetary scale \parencite{kartashova2007model,raupp2009resonant,raphaldini2022precession}. By noting their importance to numerical solutions of the $f$-plane RSWEs, we will examine the effect of linear wave errors on the dynamics of these triads. \par
The spectral evolution equation of Section \ref{section:triad_interactions} will highlight that the time evolution of triads in the RSWEs is governed by the triadic propagator of  $\mathcal{T} = \exp(i \Omega t)$. This term depends on the triadic frequency, $\Omega$, which governs the timescale of the triadic interaction. Direct- and near-resonant triads, with zero or small values of $\Omega$ respectively, lead to low-frequency dynamics, which are important in numerical simulations \parencite{chen2005resonant,Smith_Lee} and physical systems like the ocean \parencite{fu1981observations} over long time periods. Triads comprised of fast waves can still lead to a small $\Omega$ --- this means that rapid, linear, waves can construct slow timescale dynamics through the nonlinear interactions. Timestepping errors in the constituent wave frequencies will affect the discretised representation of the triad and the spectral evolution equation shows that this impacts the numerical approximation of the low-frequency dynamics. \par
We will introduce the triadic error as a new way to quantify the effect of linear dispersive errors on nonlinear RSWEs dynamics. In a Dahlquist analysis of the triadic frequency, the expected solution is a discretisation of the triadic propagator over one timestep, $\mathcal{T}_{\Delta t}$. When applying a timestepping method, an approximation of the triadic propagator ($\mathcal{T}_N$) can be constructed from stability polynomials for the individual linear waves. The triadic error then measures the discrepancy between $\mathcal{T}_N$ and $\mathcal{T}_{\Delta t}$. A key advantage of this metric is that it considers the effect of fast wave errors in slow dynamics, which can be missed when performing an analysis of solely the linear waves. \par
In part II we switch to investigating errors in full numerical models over a time interval, instead of one timestep. We achieve this by developing two new test cases to add to the existing repertoire for the RSWEs, e.g. \parencite{Williamson_tests,Galewsky_test,Zerrokaut_allen}. Test cases enable the examination of timesteppers when the spatial and temporal aspects of the model may be hard to separate, such as with Lagrangian-based timesteppers \parencite{Mengaldo_time_int}. Other papers have used test cases to rigorously analyse the effect of the spatial discretisation, such as \parencite{Staniforth_hor_grids,Weller_comp_modes}; we seek to enable similar studies over the time discretisation, like in \textcite{luan_exp_int_large_dt}. We will isolate the timestepping portion of the error by computing differences in timestepped solutions relative to a fine timestep reference solution generated within that model. This aims to fix the level of error from other components of the model, such as the spatial discretisation, and focus on the growth of timestepping error with increasing $\Delta t$. \par
Our test cases will focus on the effect of multiple timescales and important nonlinear interactions. The first test case is of a Gaussian height field perturbation, which contains a slow phase shift in its reformation due to the nonlinearity. The second test case type is triadic, where select linear waves are initialised to excite specific triadic interactions. This builds upon linear tests \parencite{Paldor_quantiative_test_case,Paldor_Matsuno}, as we let an initially linear solution evolve nonlinearly. Two variants of triadic test case highlight the importance of both direct- and near-resonant triads on the solution over long time periods. The first triadic case starts with two waves, with the nonlinear interactions constructing a third wave to complete a directly resonant triad. The second triadic case contains a cluster of slow waves, which enables an energy redistribution over rings in spectral space for the fast modes (\textit{q.v} Figure \ref{fig:TwoFSCent_waveno_space}).  \par
We now overview the layout of this work: in Section \ref{section:RSWEs} we present the dimensional RSWEs as well as a nondimensional \textit{standard equation} form with distinct scalings for the (fast) linear and (slower) nonlinear components. Part I will use this standard equation to derive the spectral evolution ODE for the RSWEs in Section \ref{section:triad_interactions}. From this, we will identify the triadic frequency and use the triadic error to compute errors in this term and the nonlinear dynamics. Section \ref{section:triad_errors} will compare triadic errors of classical timestepping methods and show how this differs from a linear dispersion analysis. In Section \ref{section:numerical_models}, we will switch to part II and overview the numerical models used in our new test cases. A spectral coefficient error metric is also described. Section \ref{section:gaussian_test_case} will then overview the Gaussian test case and Section \ref{section:triad_test_case} will present the triadic test cases. Lastly, we will conclude this paper and suggest future research directions in Section \ref{section:conclusions}. \par 

\section{The Rotating Shallow Water Equations (RSWEs)}
\label{section:RSWEs}
\subsection{Governing equations}

We consider the RSWEs on a two-dimensional, biperiodic square domain of $[0,L] \ \times \ [0,L]$ with flat topography. An $f$-plane (constant) approximation is used for the Coriolis parameter. The prognostic variables are a two-dimensional velocity ($\vec{u}$) and a measure of the elevation above the topography, such as the height of the fluid layer, $h$. A height perturbation ($\eta$) is used in the following formulation,
\begin{subequations}
\label{eq:rswes_dim}
\begin{align}
    \frac{\partial \vec{u}}{\partial t} + \left( \vec{u} \cdot \bm{\nabla} \right) \vec{u} + f \vec{\hat{k}} \times \vec{u} + g \bm{\nabla} \eta = 0, \label{eq:rswuv} \\
    \frac{\partial \eta}{\partial t} + \bm{\nabla} \cdot (\eta \vec{u}) + H_0 (\bm{\nabla} \cdot \vec{u}) = 0,
     \label{eq:rswh}
\end{align}
\end{subequations}

\noindent where $\bm{\nabla} = \left( \frac{\partial}{\partial x}, \frac{\partial}{\partial y} \right)^\text{T}$ is the two-dimensional gradient operator in the $x$-$y$ plane, $\hat{\vec{k}}$ is the unit vector in the positive $z$-direction, $f$ is the constant value of rotation, $g$ is the gravitational acceleration, and $H_0$ is the mean height of the fluid layer. Alternative measures of the elevation include a fluid depth of $D = H_0 +\eta$, which is used by Gusto and is equivalent to $h$ when there is no topography, or a geopotential of $\phi_L = g h$, as used by LFRic. \par
We will now symmetrise the RSWEs to compute the orthonormal eigenmodes (given in Appendix A). We introduce a geopotential of $\phi = \sqrt{g/H_0} ~\eta$ which ensures that the linear operator that is skew-symmetric. In the symmetrised equation set, we separate the linear and nonlinear terms to the left- and right-hand sides respectively:
\begin{subequations}
\label{eq:rswes_skew}
\begin{align}
    \frac{\partial \vec{u}}{\partial t}  + f \vec{\hat{k}} \times \vec{u} + c \bm{\nabla} \phi = - \left( \vec{u} \cdot \nabla \right) \vec{u}, \label{eq:rswuv_skew} \\
    \frac{\partial \phi}{\partial t}  + c (\bm{\nabla} \cdot \vec{u}) = - \bm{\nabla} \cdot (\phi \vec{u}),   
\label{eq:rswh_skew}
\end{align}
\end{subequations}

\noindent where $c = \sqrt{g H_0}$ is the irrotational, linear, wave frequency. The RSWEs contain three distinct eigenmodes corresponding to oscillatory dynamics, $\exp\left(i \omega_{\vec{k}}^{\alpha} t\right)$, where $\omega_{\vec{k}}^{\alpha}$ are the linear wave frequencies. These frequencies are determined from the (purely real) dispersion relation of

\begin{equation} 
\omega_{\vec{k}}^{\alpha} = \alpha \sqrt{f^2 + c^2 K^2}, ~\alpha \in \{ -1, 0, +1 \},
\label{eq:rswe_disp_rel}
\end{equation}

\noindent where $\vec{k} = [k,l]^\text{T}$ is the wavenumber vector in the $x$-$y$ plane and $K = \sqrt{k^2 + l^2}$ is the total wavenumber. The parameter $\alpha$ specifies the branch of the dispersion relation: $\alpha = 0$ corresponds to a slow mode, and $\alpha = \pm 1$ correspond to two fast inertia-gravity modes which propagate in opposite directions. The corresponding eigenvectors for these modes, $\vec{r}_{\vec{k}}^{\alpha} = \left[ ru_{\vec{k}}^{\alpha}, {rv}_{\vec{k}}^{\alpha}, r \phi_{\vec{k}}^{\alpha} \right]^\text{T}$, are given in Appendix A.

\subsection{A standard form containing a separation of timescales}
\label{subsec:RSWEs_effect_of_eps}
We now rewrite the RSWEs in what we will call the standard form. This representation highlights the presence of multiple timescales, which is a typical feature of weather and climate PDEs. Specifically, the standard form identifies a difference in the dominant timescale of the linear and nonlinear terms:
\begin{equation}
\label{eq:standard-1}
\frac{\partial \vec{U}}{\partial t} + \frac{1}{\epsilon} L \vec{U} = \mathcal{N}(\vec{U}) , ~\epsilon \in \mathbb{R}, ~\epsilon \neq 0,
\end{equation}

\noindent where $\vec{U}(t)$, is a (spatial) vector valued function $\vec{U}(t) = (U_1(t,\cdot),U_2(t,\cdot),...)$, $L$ has purely imaginary or zero eigenvalues, and $\mathcal{N}(\vec{U})$ is a nonlinear term. In practice, there could be dissipation, but as we focus on how the nonlinearity engages with the oscillatory waves, we do not include it here. The parameter $\epsilon$ in equation (\ref{eq:standard-1}) quantifies the frequency of the oscillations of the linear term. When $\epsilon$ is small, we expect fast oscillations, when $\epsilon \sim \mathcal{O}(1)$ the oscillations are slower. A regime with $\epsilon < 1$ in the RSWEs is often referred to as a state of timescale separation. 
\par
We can nondimensionalise the symmetrical RSWEs to obtain a standard form representation with an explicit separation between the timescales of $\epsilon^{-1} L \sim \mathcal{O}(\epsilon^{-1})$ and $\mathcal{N} \sim \mathcal{O}(1)$. Consider nondimensional quantities (denoted with a bar) of $\vec{u} = \mathtt{U} \vec{\bar{u}}, ~\vec{x} = \mathtt{L} \vec{\bar{x}}, ~t = \mathtt{T} \bar{t}, ~\phi= \mathtt{U} \bar{\phi}$, with $\mathtt{U}, \mathtt{L}, \mathtt{T}$ denoting characteristic velocity, length, and time scales, respectively. Using these in the symmetrised equations (\ref{eq:rswes_skew}) and letting the nondimensional time evolve on an advective scale ($\mathtt{T} = \mathtt{L}/ \mathtt{U}$) leaves:
\begin{subequations}
\begin{align}
    \frac{\partial \bar{\vec{u}}}{\partial \bar{t}}  + \frac{f \mathtt{L}}{\mathtt{U}} \vec{\hat{k}} \times \bar{\vec{u}} + \frac{c}{\mathtt{U}} \bar{\bm{\nabla}} \bar{\phi} &= - \left( \bar{\vec{u}} \cdot \bar{\bm{\nabla}} \right) \bar{\vec{u}} , \\
    \frac{\partial \bar{\phi}}{\partial \bar{t}}  + \frac{c}{\mathtt{U}} \left( \bar{\bm{\nabla}} \cdot \bar{\vec{u}} \right) &= - \bar{\bm{\nabla}} \cdot \left( \bar{\phi} \bar{\vec{u}} \right).
\end{align}
\label{eq:nonD_with_scales}
\end{subequations}

\noindent We now introduce two important nondimensional quantities of the Rossby number ($Ro$) and Froude number ($Fr$):
\begin{equation}
    Ro = \frac{\mathtt{U}}{f\mathtt{L}}, ~Fr = \frac{\mathtt{U}}{c}.
    \label{eq:nonD_numbers}
\end{equation}

\noindent The Rossby number measures the strength of the rotation and the Froude number measures the strength of the inertia-gravity force. We identify these quantities in (\ref{eq:nonD_with_scales}):

\begin{subequations}
\begin{align}
    \frac{\partial \bar{\vec{u}}}{\partial \bar{t}}  + \frac{1}{Ro} \vec{\hat{k}} \times \bar{\vec{u}} + \frac{1}{Fr} \bar{\bm{\nabla}} \bar{\phi} &= - \left( \bar{\vec{u}} \cdot \bar{\bm{\nabla}} \right) \bar{\vec{u}} , \\
    \frac{\partial \bar{\phi}}{\partial \bar{t}}  + \frac{1}{Fr} \left( \bar{\bm{\nabla}} \cdot \bar{\vec{u}} \right) &= - \bar{\bm{\nabla}} \cdot \left( \bar{\phi} \bar{\vec{u}} \right).
\end{align}
\label{eq:RSWEs_with_Ro_Fr}
\end{subequations}

\noindent When we take $Ro = Fr = \epsilon$, there is a single timescale for the linear terms and the RSWEs are in the standard form of (\ref{eq:standard-1}). This form will be used to derive the spectral evolution equation in part I of this work. For the test cases in part II, the numerical models use dimensional versions of the RSWEs, like (\ref{eq:rswes_dim}). To ensure a separation of timescales between fast oscillations and slower nonlinear interactions in this case, we choose the model parameters such that $Ro$ and $Fr$ are both small, although we do not require that they are equal.
\par
The size of $\epsilon$ has two important effects on numerical solutions to (\ref{eq:standard-1}). The first is that the $\epsilon^{-1} L$ term induces oscillations on the timescale of $\mathcal{O}(\epsilon)$, which require small timesteps if explicit numerical integrators are used. Even implicit methods, which can stably take larger timesteps, require small timesteps if accuracy is required. Second, the nonlinearity involves interactions of these faster scale oscillations that can also create low-frequency (slow) solutions of the PDE \parencite{Schochet,Embid_1996,SandersVerhulst}. \par
We can see these effects in the RSWEs through a simple one-dimensional example, obtained by setting $y$ derivatives to zero in the nondimensional RSWEs (\ref{eq:RSWEs_with_Ro_Fr}) and taking $Ro=Fr=\epsilon$. Consider an initial rest state with a Gaussian perturbation to the geopotential:
\begin{equation}
    u(x,0) = 0, ~v(x,0) = 0, ~\phi(x,0) = \exp \left( \frac{-(x-\pi)^2}{2} \right).
    \label{eq:oneD_gaussian_IC}
\end{equation}

\noindent This initialises dynamics of the Gaussian perturbation dispersing and reforming in a relatively periodic fashion. Figure \ref{fig:effectofepsongaussian} plots the geopotential height field for different linear wave speeds of $\epsilon \in \{ 1, 0.1, 0.01 \}$. The reduction of $\epsilon$ from left to right leads to faster oscillations and more reformations of the Gaussian over the same period of time. This is a consequence of the fast wave frequencies (\ref{eq:rswe_disp_rel}) becoming $\omega = \epsilon^{-1} \sqrt{1 + K^2}$ in the standard form system. In Figure 2 we compare the geopotential in the full RSWEs compared to a linear version, which is obtained by setting the right-hand side terms in (\ref{eq:RSWEs_with_Ro_Fr}) to zero. An important effect of the nonlinearity is a phase change over this duration --- in the time that nine reformations occur for the linear system, there are ten in the fully nonlinear case. Hence, the nonlinearity constructs important dynamics over a longer timescale than the linear oscillations. This begs the question of whether a timestepper can correctly identify this low-frequency nonlinear phase shift, amidst the fast oscillations, when used with a large timestep. We revisit this idea in Section \ref{section:gaussian_test_case}, where we will create a test case using a two-dimensional version of this Gaussian. \par

\begin{figure}
    \centering
     \includegraphics[scale = 0.8]{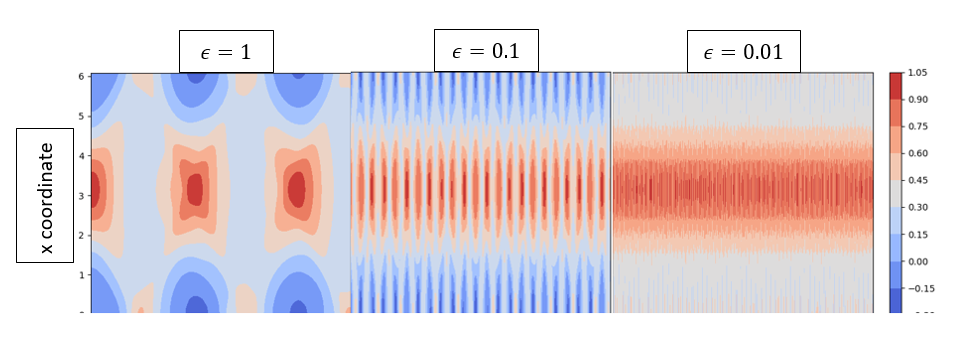}
    \caption[Gaussian reformation with epsilon]{The height field from the Gaussian initialisation of the one-dimensional RSWEs, as a function of $\epsilon$. From left to right, the regimes of $\epsilon \in \{1, 0.1, 0.01 \}$ lead to more reformations of the Gaussian height field over the same time interval of $T = 10$. This corresponds to significantly larger timestep restrictions for explicit timesteppers.}
    \label{fig:effectofepsongaussian}
\end{figure}

\begin{figure}
    \centering
     \includegraphics[scale = 0.7]{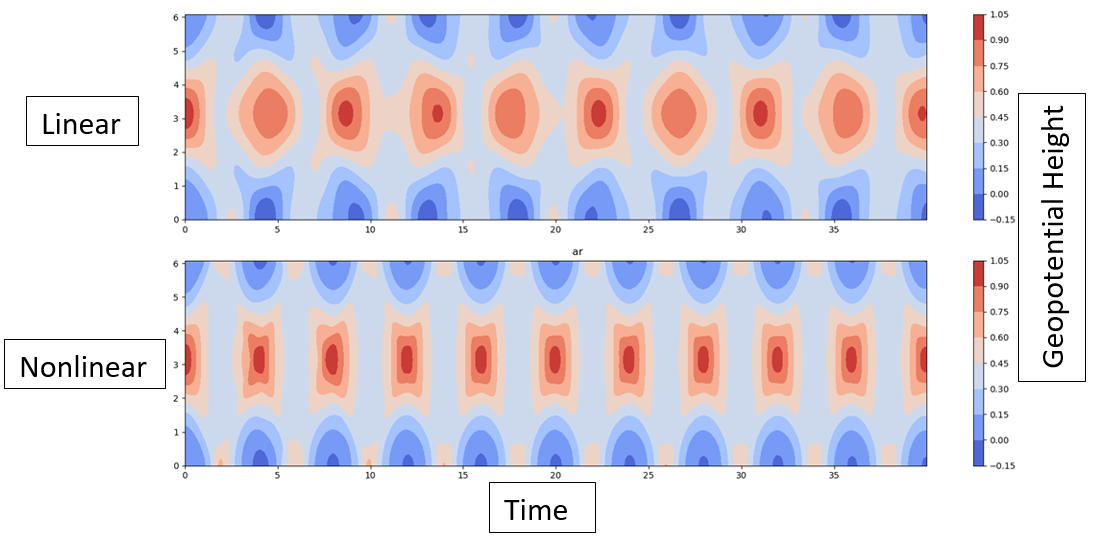}
    \caption[Gaussian reformation for linear and nonlinear RSWEs]{The height field of a reforming Gaussian perturbation, in linear (top) and nonlinear (bottom) versions of the one-dimensional RSWEs. The linear system has nine reformations of the Gaussian in the time that the nonlinear case has ten. }
    \label{fig:Gauss_lin_nonlin}
\end{figure}

\section{Triadic interactions}
\label{section:triad_interactions}

To better understand the construction of low-frequency dynamics in the RSWEs, this section reformulates the standard form PDE (\ref{eq:standard-1}) in a new basis used by \textcite{Schochet,Embid_1996,SandersVerhulst}. The resulting \textit{spectral evolution equation} represents the nonlinear time evolution of the $f$-plane RSWEs as the result of triadic interactions. Specifically, this elucidates the role of linear waves within nonlinear dynamics, and hence how linear dispersion errors can impact nonlinear accuracy. We now provide a derivation of the spectral evolution equation; see \textcite{Majda2002,Adam} for other derivations.

\subsection{Nonlinear interactions as linear combinations of linear wave frequencies}
\label{sec:mapandtriad}
We begin by defining a mapping to transform the RSWEs into a form with more time-regularity than our standard PDE (\ref{eq:standard-1}) with its separation of timescales. This is a conventional first step in the method of averaging \parencite{SandersVerhulst} and in the study of fast singular limits \parencite{Schochet,Embid_1998}. For notational convenience, we absorb $\epsilon$ into the linear term as $\mathcal{L} = \epsilon^{-1} L$: this means that $\mathcal{L} \sim \mathcal{O}(\epsilon^{-1})$ contains fast oscillations. We then define a mapping to a \textit{modulation variable}, $\vec{V}$,
\begin{equation}
\vec{U}(\vec{x},t) = e^{-t \mathcal{L}} \vec{V}(\vec{x},t),
\label{eq:map}
\end{equation}

\noindent where $\vec{x}$ denotes the spatial domain. We will compute the exponential operator in (\ref{eq:map}) using the Fourier series but this operator can also be computed in non-Fourier domains \parencite{nineteen_dubious_ways,Haut_Babb}. This $\exp(-t \mathcal{L})$ term is similar to that used in exponentially time differenced (ETD) schemes \parencite{Pope_exp_ints,cox_mathews} but is a continuous mapping instead of being discretised over a single timestep. \par
Applying the transformed variable, $\vec{V}$, in the standard equation (\ref{eq:standard-1}) leads to a different system to timestep over,
\begin{equation}
\frac{\partial \vec{V}}{\partial t} = e^{t \mathcal{L}} \mathcal{N} \left( e^{-t \mathcal{L}}\vec{V},e^{-t \mathcal{L}}\vec{V} \right).
\label{eq:modv}
\end{equation}

\noindent This new equation is termed the modulation equation, as it no longer has oscillations on $O(\epsilon)$. We will consider (\ref{eq:modv}) with a bilinear $\mathcal{N}$, as is the case with the RSWEs. \par
To identify the effect of nonlinear interactions in the modulation equation (\ref{eq:modv}) we consider a spectral space formulation. This uses Fourier expansions, which is made possible due to the periodicity of the domain. We express the modulation variable in spectral space as a Fourier expansion that is summed over all wavenumbers and modes ($\alpha $) across the two-dimensional plane $\vec{x} = [x,y]^\text{T}$,
\begin{equation}
\vec{V}(\vec{x},t) = \sum_{\vec{k} \in \mathbb{Z}^2} \sum_{\alpha \in \{-1, 0, +1 \} } \sigma_{\vec{k}}^{\alpha}(t) e^{i(\vec{k} \cdot \vec{x})} \vec{r}_{\vec{k}}^{\alpha},
\label{eq:spec_v}
\end{equation}

\noindent where $\sigma_{\vec{k}}^{\alpha}(t)$ are time-varying spectral coefficients and $\exp(i(\vec{k} \cdot \vec{x})) \vec{r}_{\vec{k}}^{\alpha}$ are eigenfunctions of the skew-Hermitian operator, $\mathcal{L}$, for the biperiodic $f$-plane. These eigenfunctions are given explicitly for the dimensional RSWEs in Appendix A. \par
To form the left-hand side of (\ref{eq:modv}) in spectral space, we take the time derivative of $\vec{V}$, which only affects the spectral coefficients:
\begin{equation}
    \frac{\partial \vec{V}}{\partial t} = \sum_{\vec{k} \in \mathbb{Z}^2} \sum_{\alpha \in \{-1, 0, +1 \} } \frac{d \sigma_{\vec{k}}^{\alpha}(t)}{dt}  
 e^{i(\vec{k} \cdot \vec{x})} \vec{r}_{\vec{k}}^{\alpha}.
 \label{eq:spec_deriv_lhs}
\end{equation}

\noexpand For the right-hand side of (\ref{eq:modv}) in spectral space, we use the Fourier expansion of (\ref{eq:spec_v}) and the eigenfrequencies ($\omega_{\vec{k}}^{\alpha}$) of the linear operator (\ref{eq:rswe_disp_rel}), for a representation of the $\exp(-t \mathcal{L}) \vec{V}$ terms in the bilinearity, 
\begin{align}
    e^{-t \mathcal{L}} \vec{V}(\vec{x},t) &= \sum_{\vec{k} \in \mathbb{Z}^2} \sum_{\alpha \in \{ -1, 0, +1 \}} \sigma_{\vec{k}}^{\alpha}(t) e^{i((\vec{k} \cdot \vec{x}) - \omega_{\vec{k}}^{\alpha} t) } \vec{r}_{\vec{k}}^{\alpha},
     \label{eq:spec_e_tL_v} \\
     \mathcal{N}(e^{-t \mathcal{L}}\vec{V},e^{-t \mathcal{L}}\vec{V}) &= \sum_{\vec{k}_a, \vec{k_b} \in \mathbb{Z}^2}\sum_{\alpha_a, \alpha_b \in \{-1, 0 +1 \}} \sigma_{\vec{k}_a}^{\alpha_a}(t) \sigma_{\vec{k_b}}^{\alpha_b}(t) e^{i ~\left( \vec{k}_a \cdot \vec{x} + \vec{k_b} \cdot \vec{x} - \left( \omega_{\vec{k}_a}^{\alpha_a} + \omega_{\vec{k_b}}^{\alpha_b} \right) t \right) } \mathcal{N} \left( \vec{r}_{\vec{k}_a}^{\alpha_a},\vec{r}_{\vec{k_b}}^{\alpha_b} \right), \label{eq:N_expansion}
\end{align}

\noindent where $\mathcal{N}(\vec{r}_{\vec{k}_a}^{\alpha_a},\vec{r}_{\vec{k_b}}^{\alpha_b})$ is a bilinear function of the eigenvectors. We project this expression of the nonlinearity back onto the eigenbasis of the RSWEs to identify interactions between waves. This uses the complex inner product, defined as $\left< \vec{a}, \vec{b} \right> = \vec{a} \cdot \vec{b}^*$, with $\cdot$ representing the dot product and $^*$ denoting complex conjugation. Computing $\left< (\ref{eq:N_expansion}), \vec{r}_{\vec{k}}^{\alpha} \right>$:

\begin{equation}
\begin{split}
\label{eq:N_projected}
\mathcal{N}\left(e^{-t \mathcal{L}}\vec{V},e^{-t \mathcal{L}}\vec{V}\right) =& \\
\sum_{\vec{k} \in \mathbb{Z}^2} \sum_{\alpha \in \{-1, 0, +1 \} } 
\sum_{\vec{k}_a, \vec{k_b} \in \mathbb{Z}^2}\sum_{\alpha_a, \alpha_b \in \{-1, 0, +1 \} }& \sigma_{\vec{k}_a}^{\alpha_a}(t) \sigma_{\vec{k_b}}^{\alpha_b}(t) e^{i ~\left( (\vec{k}_a +  \vec{k_b} ) \cdot \vec{x} - (\omega_{\vec{k}_a}^{\alpha_a} + \omega_{\vec{k_b}}^{\alpha_b})t \right) } C_{\mathbf{k}_{a},\mathbf{k}_{b},\mathbf{k}}^{\alpha_{a},\alpha_{b},\alpha} \vec{r}_{\vec{k}}^{\alpha}.
\end{split}
\end{equation}

\noindent $C_{\mathbf{k}_{a},\mathbf{k}_{b},\mathbf{k}}^{\alpha_{a},\alpha_{b},\alpha} = \left< \mathcal{N} \left( \vec{r}_{\vec{k}_a}^{\alpha_a},\vec{r}_{\vec{k_b}}^{\alpha_b} \right) , \vec{r}_{\vec{k}}^{\alpha} \right> $ are interaction coefficients that can be found in \textcite{Majda2002}. Applying the positive matrix exponential to (\ref{eq:N_projected}) forms the right-hand side of the modulation equation (\ref{eq:modv}) in spectral space:
\begin{equation} 
\begin{split}
e^{t \mathcal{L}}
\mathcal{N}(e^{-t \mathcal{L}}\vec{V},e^{-t \mathcal{L}}\vec{V}) 
=& \\
\sum_{\vec{k} \in \mathbb{Z}^2} \sum_{\alpha \in \{-1, 0, +1 \} } \sum_{\vec{k}_a, \vec{k_b} \in \mathbb{Z}^2}\sum_{\alpha_a, \alpha_b \in \{-1, 0, +1 \} } &
\sigma_{\vec{k}_a}^{\alpha_a}(t) \sigma_{\vec{k_b}}^{\alpha_b}(t) 
e^{i ~\left( (\vec{k}_a + \vec{k_b}) \cdot \vec{x} - \Omega_{\mathbf{k},\mathbf{k}_{a},\mathbf{k}_{b}}^{\alpha,\alpha_{a},\alpha_{b}} t \right) }  C_{\mathbf{k}_{a},\mathbf{k}_{b},\mathbf{k}}^{\alpha_{a},\alpha_{b},\alpha}  \vec{r}_{\vec{k}}^{\alpha}.
 \label{eq:spec_deriv_rhs}
 \end{split}
\end{equation}

\noindent We see that the following quantity forms in the matrix exponential:
\begin{equation}
\Omega_{\mathbf{k},\mathbf{k}_{a},\mathbf{k}_{b}}^{\alpha,\alpha_{a},\alpha_{b}} = \omega_{\mathbf{k}_{a}}^{\alpha_{a}} + \omega_{\mathbf{k}_{b}}^{\alpha_{b}} - \omega_{\mathbf{k}}^{\alpha}.
\label{eq:triad_freq}
\end{equation}

\noindent This is the \textit{triadic frequency} and is a linear combination of frequencies of the linear operator, $\mathcal{L}$, for a triad of specific wavenumbers and mode types. This frequency measures the timescale of the triadic interaction and so is important for the resulting dynamics. Equation (\ref{eq:triad_freq}) shows that dispersive errors in the linear wave frequencies of $\omega_{\vec{k}}^{\alpha}$ will appear in the numerical approximation to $\Omega_{\mathbf{k},\mathbf{k}_{a},\mathbf{k}_{b}}^{\alpha,\alpha_{a},\alpha_{b}}$. This can impact the correct representation of nonlinear dynamics and is the key idea behind our new triadic error. \par
Now that we have left- and right-hand side expressions for our modulation equation in spectral space (\ref{eq:spec_deriv_lhs} and \ref{eq:spec_deriv_rhs}), we project these back onto the complex eigenfunctions to obtain the spectral evolution equation. This is achieved with complex inner products of $\left< \cdot , \sum_{\vec{k} \in \mathbb{Z}^2} \sum_{\alpha} e^{i(\vec{k} \cdot \vec{x})} \vec{r}_{\vec{k}}^{\alpha} \right>$ and integration over the domain for each $\vec{k}$ and $\alpha$. This results in an ODE in time for the evolution of each spectral coefficient:
\begin{align}
\label{eq:sigmaeqn}
\frac{d \sigma_{\vec{k}}^{\alpha}}{d t} =  
    \sum_{\substack{\vec{k}_a, \vec{k}_b \in \mathbb{Z}^2 \\ \vec{k}_a + \vec{k}_b = \vec{k}}} \sum_{\alpha_a, \alpha_b \in \{ -1, 0, +1 \} } \sigma_{\vec{k}_a}^{\alpha_a}(t) \sigma_{\vec{k}_b}^{\alpha_b}(t)  C_{\mathbf{k}_{a},\mathbf{k}_{b},\mathbf{k}}^{\alpha_{a},\alpha_{b},\alpha}  e^{-i \Omega_{\mathbf{k}_{a},\mathbf{k}_{b},\mathbf{k}}^{\alpha_{a},\alpha_{b},\alpha} t  }  .
\end{align}

\noindent The only wavenumber combinations that remain in this ODE are those that satisfy the triadic constraint of
\begin{equation}
    \label{eq:triadic_constraint}
    \vec{k} = \vec{k}_a + \vec{k}_b,
\end{equation}

\noindent where $\vec{k}_a, \vec{k}_b$ are wavenumbers of the incoming waves in the triad and $\vec{k}$ is the wavenumber of the outgoing wave. This is due to the orthonormality of the eigenbasis, as
\begin{equation}
    \left< e^{i \left( (\vec{k}_a + \vec{k}_b) \cdot \vec{x} \right) } \vec{r}_{\vec{k}}^{\alpha}, e^{i(\vec{k} \cdot \vec{x})} \vec{r}_{\vec{k}}^{\alpha} \right> = \begin{cases}
    1, &\text{if } \vec{k} = \vec{k}_a + \vec{k}_b, \\
    0, &\text{if } \vec{k} \neq \vec{k}_a + \vec{k}_b.
    \end{cases}
\end{equation}

The spectral evolution equation (\ref{eq:sigmaeqn}) shows that the nonlinear evolution of each spectral coefficient, $\sigma_{\mathbf{k}}^\alpha(t)$, depends on other spectral coefficients, an interaction coefficient, and an oscillatory term with the triadic frequency of $\Omega_{\mathbf{k}_{a},\mathbf{k}_{b},\mathbf{k}}^{\alpha_{a},\alpha_{b},\alpha}$ in its argument. This ODE represents the same dynamics as the standard form PDE (\ref{eq:standard-1}) but views the time evolution as the direct result of triadic interactions. As such, it provides an alternative approach for analysing nonlinear time discretisation errors in numerical solutions to the $f$-plane RSWEs. (\ref{eq:sigmaeqn}) is similar to the expression found in \textcite{Embid_1996}; the difference is that we have not considered asymptotic solutions or performed any averaging, so retain all nonlinear interactions over a wide range of triadic frequencies. \par
The term which governs the time evolution of the triads in the spectral evolution equation is

\begin{equation}
    \mathcal{T} \left( \Omega_{\mathbf{k},\mathbf{k}_{a},\mathbf{k}_{b}}^{\alpha,\alpha_{a},\alpha_{b}},t \right) = e^{i \Omega_{\mathbf{k},\mathbf{k}_{a},\mathbf{k}_{b}}^{\alpha,\alpha_{a},\alpha_{b}} t}.
    \label{eq:triad_prop}
\end{equation}

\noindent which we will call the \textit{triadic propagator}. We define $\mathcal{T}$ with a positive complex exponential, as the difference in sign from that in (\ref{eq:sigmaeqn}) change does not affect the analysis. The triadic propagator represents a time-varying quantity for a constant triadic frequency. When $\Omega_{\mathbf{k},\mathbf{k}_{a},\mathbf{k}_{b}}^{\alpha,\alpha_{a},\alpha_{b}}$ is zero or small, the triadic propagator will be constant or exhibit slow oscillations. These triads are classified as being direct- or near-resonant, respectively, and will generate low-frequency dynamics that are important over long time periods. The only triadic interactions that remain in the rapidly oscillating $\epsilon \rightarrow 0$ limit in the standard equation are the direct resonances \parencite{Embid_1996}. However, for finite values of $\epsilon$, near-resonances also provide a significant contribution to the dynamics \parencite{Newell_rossby,smith1999transfer} so they are important in this work. \par
As well as affecting the dynamics through the triadic propagator, the size of the triadic frequencies affects errors in a timestepped version of the spectral evolution equation (\ref{eq:sigmaeqn}). This was investigated by \textcite{Adam}, who derived the following error bound for a second-order timestepping scheme in the $f$-plane RSWEs,
\begin{equation}
    || \vec{y}(t) - \vec{y}_{\Delta t}(t) || \leq C M \left(\Delta t\right)^3 \max_{n \in \mathbb{N}} \left| \Omega_n \right|^2,
    \label{eq:adam_error_bound}
\end{equation}

\noindent where $C$ is a constant and $M$ is a bound on the size of the nonlinearity. $\Omega_n$ are the triadic frequencies of the resolvable interactions in the discretised domain, indexed in increasing order such that $|\Omega_{n}| \leq |\Omega_{n+1}|$. This error bound (\ref{eq:adam_error_bound}) shows that triadic interactions and their frequencies can be used to analyse large timestep errors. With this as motivation, we instead take a different approach and compute errors in a numerical representation of the triadic frequencies, $\Omega_{\mathbf{k},\mathbf{k}_{a},\mathbf{k}_{b}}^{\alpha,\alpha_{a},\alpha_{b}}$; we now introduce the triadic error to achieve this.

\subsection{A numerical triadic error}
To estimate errors in the nonlinear interaction frequencies we will use linear dispersion theory, which quantifies timestepping errors in representing linear wave frequencies. Such linear dispersion analyses have provided intuition in many wave propagation PDEs, e.g. \parencite{trefethen1982group,wingate_timestep_alpha,long2011numerical}. Our use of this theory to examine nonlinear interactions is a novel application. \par 
Linear dispersion errors can be identified by applying a timestepper to the complex-valued Dahlquist test equation \parencite{iserlies}. We consider the purely imaginary version of
\begin{equation}
    \label{eq:Dahlquist}
    \frac{d\vec{U}}{dt} = i \omega_{\mathbf{k}}^{\alpha} \vec{U},
\end{equation}

\noindent where an explicit factor of $i$ is extracted to leave a real-valued $\omega_{\mathbf{k}}^{\alpha}$, as per the RSWEs (\ref{eq:rswe_disp_rel}). (\ref{eq:Dahlquist}) is referred to as the oscillation equation \parencite{Durran} and its solution over a given timestep is an exponential mapping:
\begin{equation}
    \label{eq:Dahlquist_sol}
    \vec{U}(t + \Delta t) =  e^{i \omega_{\mathbf{k}}^{\alpha} \Delta t} \vec{U}(t).
\end{equation}

Applying a timestepping method in the oscillation equation will generate an approximate solution after one timestep, $\vec{U}_N(t + \Delta t)$. The mapping from $\vec{U}(t)$ to $\vec{U}_N(t + \Delta t)$ is the stability polynomial, $P$, of a timestepping scheme,
\begin{equation}
    \vec{U}_N(t + \Delta t) = P(W_{\mathbf{k}}^{\alpha}) \vec{U}(t).
    \label{eq:stability_poly}
\end{equation}

\noindent where $W_{\mathbf{k}}^{\alpha} = i \omega_{\mathbf{k}}^{\alpha} \Delta t$. Comparing (\ref{eq:Dahlquist_sol}) and (\ref{eq:stability_poly}) shows that $P(W_{\mathbf{k}}^{\alpha})$ represents a numerical method's approximation to $\exp(W_{\mathbf{k}}^{\alpha})$. \par
Now, we apply a Dahlquist analysis on the nonlinear interactions that appear in the RSWEs. Let us consider the oscillation equation for a triad of frequency $\Omega_{\mathbf{k}_{a},\mathbf{k}_{b},\mathbf{k}}^{\alpha_{a},\alpha_{b},\alpha}$:
\begin{equation}
    \label{eq:Dahlquist_om}
    \frac{d\vec{U}}{dt} = i \Omega_{\mathbf{k}_{a},\mathbf{k}_{b},\mathbf{k}}^{\alpha_{a},\alpha_{b},\alpha} \vec{U}.
\end{equation}

\noindent This equation has an analytical solution of
\begin{equation}
    \label{eq:dahlquist_om_sol}
    \vec{U}(t + \Delta t) =  e^{i \Omega_{\mathbf{k}_{a},\mathbf{k}_{b},\mathbf{k}}^{\alpha_{a},\alpha_{b},\alpha} \Delta t} \vec{U}(t).
\end{equation}

\noindent This time evolution operator is a discretisation over one timestep of the triadic propagator (\ref{eq:triad_prop}). We denote the discretised version $\mathcal{T}_{\Delta t}$ and decompose it into contributions from the individual linear waves:
\begin{align}
\begin{split}
    \mathcal{T}_{\Delta t} \left( \Omega_{\mathbf{k}_{a},\mathbf{k}_{b},\mathbf{k}}^{\alpha_{a},\alpha_{b},\alpha},\Delta t \right) = e^{i \Omega_{\mathbf{k}_{a},\mathbf{k}_{b},\mathbf{k}}^{\alpha_{a},\alpha_{b},\alpha} \Delta t} &= \exp \left(i \left( \omega_{\mathbf{k}_{a}}^{\alpha_{a}} + \omega_{\mathbf{k}_{b}}^{\alpha_{b}} - \omega_{\mathbf{k}}^{\alpha} \right) \Delta t \right) \\
    &= 
    \exp \left( W_{\vec{k}_a}^{\alpha_a} \right) \exp \left( W_{\vec{k}_b}^{\alpha_b} \right) \exp \left( - W_{\vec{k}}^{\alpha} \right),
\label{eq:triad_prop_split}   
\end{split}
\end{align}

\noindent We can construct a numerical approximation to $\mathcal{T}_{\Delta t}$ by using stability polynomials for each of the linear waves:
\begin{equation}
    \mathcal{T}_N \left( W_{\vec{k}_a}^{\alpha_a}, W_{\vec{k}_b}^{\alpha_b}, W_{\vec{k}}^{\alpha} \right) =  P\left(W_{\vec{k}_a}^{\alpha_a}\right) P\left(W_{\vec{k}_b}^{\alpha_b}\right) P\left(- W_{\vec{k}}^{\alpha}\right).
    \label{eq:num_triad_prop}
\end{equation}

\noindent We then define the triadic error as the difference between the true discretised triadic propagator, $\mathcal{T}_{\Delta t}$, and the numerical approximation, $\mathcal{T}_N$:
\begin{equation}
    E \left( \vec{k}, \vec{k}_a, \vec{k}_b, \alpha, \alpha_a, \alpha_b, \Delta t \right) = \left\Vert \mathcal{T}_{\Delta t} \left( \Omega_{\mathbf{k}_{a},\mathbf{k}_{b},\mathbf{k}}^{\alpha_{a},\alpha_{b},\alpha},\Delta t \right) - \mathcal{T}_N \left( W_{\vec{k}_a}^{\alpha_a}, W_{\vec{k}_b}^{\alpha_b}, W_{\vec{k}}^{\alpha} \right) \right\Vert .
    \label{eq:triadic_error}
\end{equation}

\noindent The triadic error can be computed for any valid triadic combination of three linear waves. It is measured over \textit{one timestep}, which is important when determining the order of accuracy of timestepping methods, as we shall do in the next section. \par
There are four key differences between the triadic error and linear dispersion analyses of the individual waves:
\begin{enumerate}
    \item The triadic error quantifies time discretisation errors in a component of the \textit{nonlinear} interactions.
    \item It simultaneously quantifies errors from the linear and nonlinear terms of the RSWEs through the spectral evolution equation (\ref{eq:sigmaeqn}) as opposed to analysing $\mathcal{L}$ and $\mathcal{N}$ independently in the standard form PDE (\ref{eq:standard-1}).
    \item The importance of errors on long timescale solutions is determined by the frequency of the nonlinear interaction, $|\Omega|$, as opposed to the speed of the linear waves, $|\omega|$. Hence, the triadic error reflects how errors in fast oscillations can impact the accuracy of slow dynamics.
    \item The triadic error considers both the magnitude of the linear dispersion errors and how these are composed within the temporal discretisation. This generates a different relative comparison between timesteppers compared to individual linear dispersion analyses.
\end{enumerate}

\subsection{Triadic interactions in the discretised RSWEs}
\label{subsec:rswe_triad_interactions}
To compute triadic errors in a numerical model, we first need to ascertain the set of triads that can contribute to the dynamics. As opposed to a continuous physical system, a computational grid only permits a finite number of triads. A valid triad not only must satisfy the triadic condition on the wavenumbers (\ref{eq:triadic_constraint}), but the third wavenumber needs to remain within the domain of resolvable wavenumbers, as detailed in \textcite{Smith_Lee}. The number of permissible triads rapidly grows with the domain size, which means that modelling all the interactions in (\ref{eq:sigmaeqn}) can be extremely computationally expensive. Hence, reduced models that only compute the dominant interactions are often used in practice, such as in \textcite{Smith_Lee,Adam}. As the direct- and near-resonances are the dominant contributors to long time dynamics, we select these by introducing a \textit{cut-off frequency}, $\Omega_C$. This determines a \textit{dominant subset}, $S_D$, which contains the $N$ triads with a frequency below or equal to this cut-off value:
\begin{equation}
    S_D := \{ \Omega_n(\vec{k}_a, \vec{k}_b, \vec{k}, \alpha_{a}, \alpha_{b}, \alpha), n = \{ 1,2, .. ,N \} \ | \ |\Omega_n| \leq |\Omega_{n+1}|, |\Omega_n| \leq \Omega_C \}.
    \label{eq:dom_triad_set}
\end{equation}

\noindent For instance, setting $\Omega_C = 0$ will leave only directly resonant triads in $S_D$, whereas $\Omega_C = \epsilon$ also includes near-resonant triads based on a classical definition \parencite{Newell_rossby, Smith_Lee}. The average triadic error for a given numerical timestepper is computed as the mean triadic error over all interactions in the dominant subset, 
\begin{equation}
    E(\Omega_C, \Delta t) = \frac{1}{N} \sum_{\Omega_n \in S_D} E \left( \vec{k}, \vec{k}_a, \vec{k}_b, \alpha, \alpha_a, \alpha_b, \Delta t \right) .
    \label{eq:num_triad_error_metric}
\end{equation}

A key feature to determining the triads contained in $S_D$ is that $\Omega_{\mathbf{k},\mathbf{k}_{a},\mathbf{k}_{b}}^{\alpha,\alpha_{a},\alpha_{b}}$ is not only a function of the wavenumbers, but also of the interacting wave modes, $\alpha$. With three different modes in the RSWEs, $\alpha \in \{ -1, 0, 1 \}$, there are  $3^3=27$ permutations of three linear waves. Many of these permutations are equivalent due to the arbitrary indexing of the incoming waves and the two fast modes being identical in all but sign --- this leaves six main interaction types. However, types \textit{ii}, \textit{iii}, \textit{iv}, and \textit{v} have a slight variation in the distribution of frequencies depending on the mode type of the outgoing wave relative to the incoming ones. For example, a slightly different distribution of frequencies occurs for the type \textit{ii} interaction (two fast waves and one slow), when the outgoing wave is slow ($ \alpha_a = \pm 1, \alpha_b = \mp 1, \alpha = 0$) versus when one of the incoming waves is slow (e.g. $\alpha_a = 0, \alpha_b = \mp 1, \alpha = \mp 1$). This difference is a result of the discretised spatial domain and additional \textit{a} and \textit{b} subtypes are introduced to account for this. Table 1 overviews these ten interaction types in order of increasing triadic frequency, with $S$ denoting the slow ($\alpha=0$) mode and $\pm$ denoting the fast ($\alpha=\pm1$) modes.  \par

\begin{table}
\caption{A table of the 27 permutations of triadic interactions in the RSWEs, giving the range of triadic frequencies with $\epsilon=0.1, N_x = N_y = 32$ grid points. The interaction types are ordered by increasing $\Omega$ values; the smaller numeral types contain triads with more dominant contributions over long time periods. The involved modes gives the modes of three waves, using $S$ for the slow $\alpha=0$ mode and $\pm$ for the fast $\alpha=\pm1$ modes. Types \textit{ii}, \textit{iii}, \textit{iv}, and \textit{v} have subtypes as a result of different frequencies resulting when the mode of the outgoing wave changes (e.g. for type \textit{ii}, $\pm S \pm and S \pm \pm$ vs $\pm \mp S$). We also use shorthand acronyms to refer to specific interaction types.}
\begin{center}
\begin{tabular}{ |c|c||c|c|c|c| }
\hline
Interaction Type &
Acroynm &
Involved Modes &
Count &
$\Omega$ range\\ 
\hline\hline
\hline
\textit{i} & SSS & $SSS$ & 1 & [0,0] \\ 
\hline
\textit{ii-a}& FSF & $\pm \mp S$ & 2 & [0,216.5]\\ 
\hline
\textit{ii-b} & FSF & $ \pm S \pm, S \pm \pm$ & 4 & [0,216.5] \\
\hline
\textit{iii-a} & FFF & $\pm \pm \pm$ & 2 &  [0.66,421.77]\\ 
\hline
\textit{iii-b} & FFF & $ \pm \mp \pm, \pm \mp \mp$ & 4 & [0.66,442.99] \\
\hline
\textit{iv-a} & SFS & $SS\pm$ & 2 & [10, 226.50]\\ 
\hline
\textit{iv-b} & SFS & $ \pm S S, S \pm S$ & 4 & [10, 226.50] \\
\hline
\textit{v-a} & - & $\pm \pm S$ & 2 & [20,439.09]\\ 
\hline
\textit{v-b} & - & $ S \pm \mp, \pm S \mp$ & 4  & [20, 452.50] \\
\hline
\textit{vi} & - & $\pm \pm \mp $ & 2 & [30, 547.12]\\
\hline
\end{tabular}
\end{center}
\label{table:triad_list}
\end{table}

Interaction type \textit{i}, the slow-slow-slow (SSS) interaction, involves all zero frequency modes on the $f$-plane. As the selected numerical methods all correctly model zero dispersion relations with $P(0)=1$, these will not be analysed in this work. Interaction type \textit{ii}, the fast-slow-fast (FSF) interaction, has direct resonance contributions, as well as a proportion of near-resonances, depending on $\Omega_C$. Interaction type \textit{iii}, a fast-fast-fast (FFF) interaction, can never construct direct resonances in a discrete domain, but contains small $\Omega$ near-resonances. This interaction type shows how combinations of entirely fast waves can construct low-frequency dynamics. Interaction type \textit{vi} is also comprised solely of fast modes, but the combination of the signs is such that only much larger triadic frequencies are present. \par
Our computation of triadic errors in the following section will consider interaction types \textit{ii} and \textit{iii} of both subtypes \textit{a} and \textit{b}. Interaction types \textit{iv}, \textit{v}, or \textit{vi} are not included as no triads of these types appear in $S_D$ for our choice of spatial discretisation and $\Omega_C < 10$. 

\section{Triadic error analysis of timesteppers in the $f$-plane RSWEs}
\label{section:triad_errors}
We now use the triadic error metric to compare some numerical timestepping schemes. A selection of explicit and implicit methods is provided, but any scheme with a known stability polynomial can be analysed in this manner. Individual analyses (Sections \ref{subsec:RK}--\ref{subsec:ETD}) are provided before comparing a selection of these in Section \ref{subsec:triadic_comps}. Additionally, a comparison of linear dispersion errors with the triadic error will highlight the differences between these two metrics. We again note that the triadic error is computed over one timestep, so the timestepping orders of accuracy will reflect this. \par
Measures of the triadic errors in this section use $\epsilon = 0.1$ and a spatial discretisation of $N_x=N_y=32$ grid points in the $x$ and $y$ dimensions. Triadic errors are computed using (\ref{eq:num_triad_error_metric}) and plotted as a function of timestep size at two cut-off frequencies. The first is $\Omega_C = \epsilon = 0.1$, which, with the chosen level of spatial discretisation, only permits direct interaction type \textit{ii} (FSF) triads to be present in $S_D$. The second cut-off frequency of $\Omega_C = 5$ permits a large number of near-resonant interactions, of both mode types \textit{ii} (FSF) and \textit{iii} (FFF). Logarithmic axes are used in the triadic error plots to highlight order-of-magnitude differences between the schemes. \par

\subsection{Explicit RK methods}
\label{subsec:RK}
Runge-Kutta (RK) schemes are multistage methods, as they obtain information at intermediate times to form an average function evaluation for each timestep. We consider explicit $K$-order and $K$-stage RK schemes, whose stability polynomial is a truncation of a complex exponential expansion \parencite{iserlies}. Such RK timesteppers commit an $\mathcal{O}(\Delta t)^{K+1}$ error in the dispersion relation of any linear frequency over one timestep:

\begin{equation}
    P_{\text{RK,K}} = \sum_{k=0}^{K} \frac{(i \omega \Delta t)^k}{k!} = \sum_{k=0}^{K} \frac{W^k}{k!}.
    \label{eq:P_RK}
\end{equation}

For a non-direct FSF (type \textit{ii-b}) triad, it can be shown that combining stability polynomials in $\mathcal{T}_N$ leads to a triadic error of the same $\mathcal{O}(\Delta t)^{K+1}$ order. However, for the direct resonances, an even-order RK method has an increased $\mathcal{O}(\Delta t)^{K+2}$ order of accuracy in the triadic error. Proofs of these error levels are provided in Theorems 1 and 2 in Appendix B. \par
As a higher-order RK scheme more accurately approximates the linear wave frequencies, it would be expected to have a lower error when resolving a given triad. Figure \ref{fig:RK_methods_triadic_error} indeed shows smaller triadic errors for the higher-order RK schemes. $\Omega_C=0.1$ has a larger accuracy difference at the smallest timesteps between RK2 and RK1, and RK4 and RK3; this highlights the higher order of accuracy for direct resonances with the even-order schemes. This difference reduces with increasing $\Omega_C$, as the inclusion of many non-resonant interactions --- which have $\mathcal{O}(\Delta t)^{K+1}$ triadic errors for any RK method --- causes RK2 and RK4 to commit relatively larger errors. \par

\begin{figure}
     \centering
     \begin{subfigure}[b]{0.48\textwidth}
         \centering
         \includegraphics[width=0.95\textwidth]{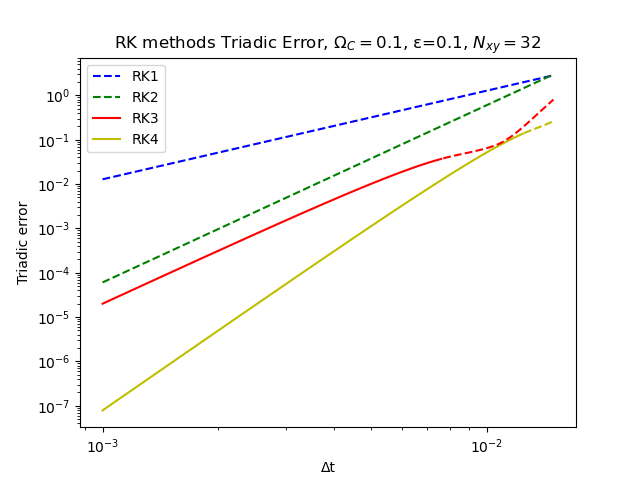}
         \caption{$\Omega_C = 0.1$}
     \end{subfigure}
     \hfill
     \begin{subfigure}[b]{0.48\textwidth}
         \centering
         \includegraphics[width=0.95\textwidth]{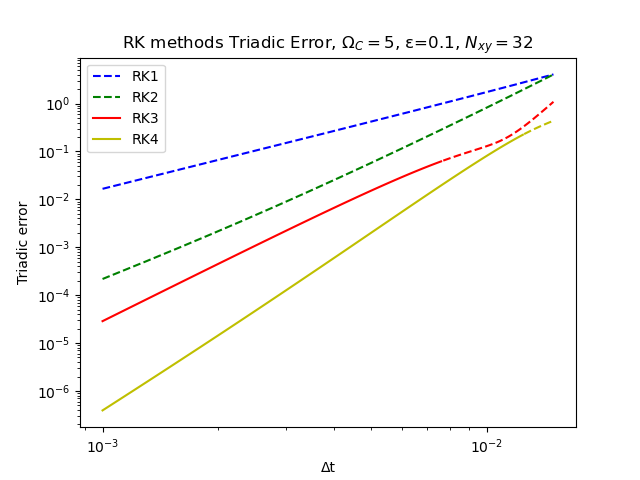}
         \caption{$\Omega_C = 5$}
     \end{subfigure}
        \caption{Triadic errors of the first four explicit RK methods at two cut-off frequencies. Stable timesteps are given by the solid lines, with dashed lines being regions above the timestep limit from the oscillations. RK1 and RK2 are shown to illustrate the effect of scheme order on the triadic errors but are in fact unconditionally unstable for purely oscillatory behaviour.}
    \label{fig:RK_methods_triadic_error}
\end{figure}

\subsection{Generalised Euler (alpha) methods}
These methods use a composition of implicit and explicit timestepping information. The parameter, $\alpha \in [0,1]$, quantifies the amount of implicit information used from the future time location, as opposed to the current time location; $\alpha = 0$ reduces to the fully explicit forwards Euler method, while $\alpha=1$ reduces to the fully implicit backwards Euler method. The stability polynomial, as a function of $\alpha$, is given in \textcite{Durran},

\begin{equation}
    P_{\alpha} = \frac{1 + (1 - \alpha)W}{1 - \alpha W} .
    \label{eq:P_gen_euler}
\end{equation}

\noindent These regimes have a switch in stability behaviour at the balanced case of $\alpha=0.5$, which is known as the trapezoidal method. For $\alpha \in [0,0.5)$, this scheme is unconditionally unstable, while it is unconditionally stable when $\alpha \in [0.5,1]$. \par
Figure \ref{fig:Euler_methods_triadic_error} compares triadic errors for selected stable alpha schemes, ranging from trapezoidal to backwards Euler. The trapezoidal method has no error at $\Omega_C = 0.1$, as it generates no triadic error for directly resonant triads --- this is proved in Theorem 3 in Appendix B. With the introduction of many near-resonances in Figure 5b, the trapezoidal method does generate a significant triadic error, but remains the optimal generalised Euler scheme. This is consistent with $\alpha=0.5$ being the only selection with a second-order, as opposed to first-order, dispersion relation error \parencite{Durran}. The greater the deviation of $\alpha$ from 0.5, the greater the triadic error that is committed, with backwards Euler generating the most for any timestep size. \par

\begin{figure}
     \centering
     \begin{subfigure}[b]{0.48\textwidth}
         \centering
         \includegraphics[width=0.95\textwidth]{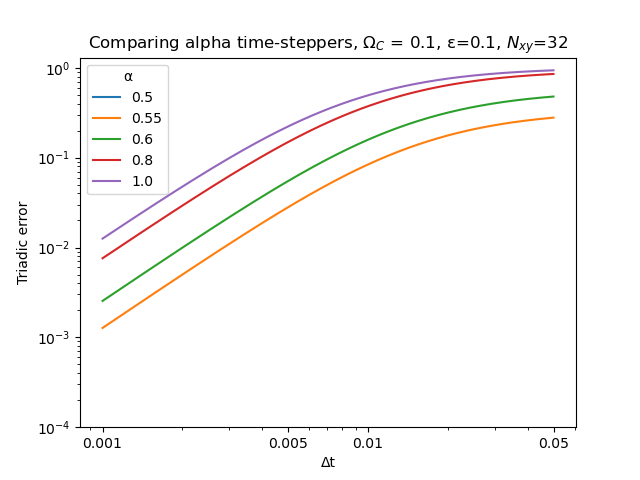}
         \caption{$\Omega_C = 0.1$}
     \end{subfigure}
     \hfill
     \begin{subfigure}[b]{0.48\textwidth}
         \centering
         \includegraphics[width=0.95\textwidth]{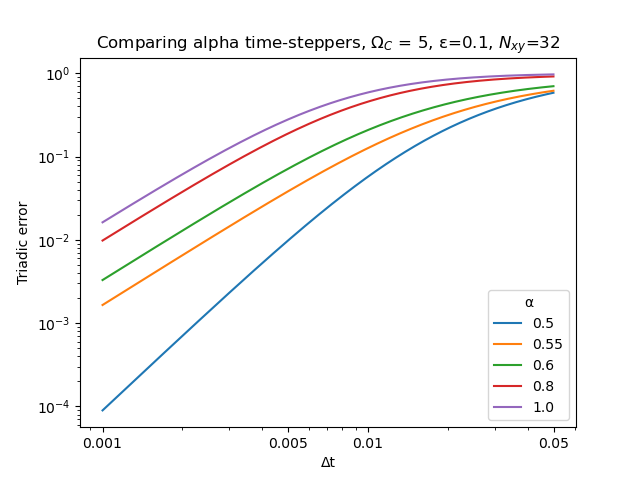}
         \caption{$\Omega_C = 5$}
     \end{subfigure}
        \caption{Triadic errors for stable alpha methods, of $\alpha \in \{ 0.5,0.55,0.6,0.8,1  \}$. There is an increasing triadic error with larger $\alpha$ for both $\Omega_C = 0.1$ and $\Omega_C = 5$. As the trapezoidal method generates zero error for direct resonances, it is not shown in (a). This scheme remains optimal when including near resonances in (b).}
    \label{fig:Euler_methods_triadic_error}
\end{figure}

\subsection{AB methods}
Explicit Adams-Bashforth (AB) methods are termed multistep, as they reuse gradients from previously computed timesteps. The corresponding stability polynomials require solving for the roots, $\lambda$ \parencite{iserlies}. The stability polynomial for the AB3 method is:

\begin{equation}
P_{\text{AB3}}: \ \ \ \ \ \lambda^3 - \left( 1 + \frac{23}{12} W \right) \lambda^2 + \frac{16}{12} W \lambda - \frac{5}{12} W = 0 .
\end{equation}

Most AB schemes have multiple roots to their stability polynomial, with some corresponding to unphysical computational modes. Our analysis uses triads consisting only of physical modes. This means that the AB3 results in Section \ref{subsec:triadic_comps} are for the best-case scenario.

\subsection{TR-BDF2}
Backwards differentiation formulae (BDF) are implicit multistep schemes, designed to ensure large stability regions \parencite{Durran}. The second order scheme, TR-BDF2, is L-stable due to using the trapezoidal rule in its first step \parencite{TR-BDF2}. Its stability polynomial is:
\begin{equation}
P_{\text{TR-BDF2}} = \frac{1 +\frac{5}{12}W}{1 - \frac{7}{12}W + \frac{1}{12}W^2} .
\label{eq:P_TR_BDF2}
\end{equation}

\subsection{Exponential time differencing schemes}
\label{subsec:ETD}
Exponential time differencing (ETD) schemes were developed to overcome numerical stiffness, resulting from a small $\epsilon$ in the standard equation (\ref{eq:standard-1}), by using the exact linear operator \parencite{Pope_exp_ints,cox_mathews}. This matrix exponential integrating factor is similar to the modulation variable mapping (\ref{eq:map}) except that it is discretised over each timestep as $\exp(-\Delta t \mathcal{L})$. To perform a timestep, an ETD scheme solves
\begin{equation}
    \label{eq:exp_int}
   \vec{U}(t_n + \Delta t) = e^{-\Delta t \mathcal{L}} \vec{U}(t_n) + e^{-\Delta t \mathcal{L}} \int_{t_n}^{t_n+ \Delta t} e^{(\tau - t_n) \mathcal{L}} \mathcal{N}(\vec{U}(\tau)) d \tau .
\end{equation}

\noindent Higher-order ETD methods use more accurate approximations of the integral of the nonlinearity. Examples of exponential integrator schemes include Runge-Kutta, Rosenbrock, and multistep implementations \parencite{Hochburck_Ostermann_exp_ints}. In this paper, we will examine the ETD-RK2 scheme of \textcite{cox_mathews}. \par
Due to inherently possessing an exact representation of the linear waves, the stability polynomials of ETD schemes are exact solutions to the oscillatory Dahlquist equation,
\begin{equation}
    P_{\text{ETD}} = e^{\Delta t \mathcal{L}} = e^{i \omega \Delta t} .
\end{equation}

\noindent Thus,
\begin{equation}
    \mathcal{T}_N = \exp \left(i \omega_{\mathbf{k}_{a}}^{\alpha_{a}} \Delta t\right) \exp \left(i \omega_{\mathbf{k}_{b}}^{\alpha_{b}} \Delta t \right) \exp \left(-i \omega_{\mathbf{k}}^{\alpha} \Delta t \right) = e^{i \Omega \Delta t} = \mathcal{T}_{\Delta t},
\end{equation}

\noindent and so, 
\begin{equation}
    E_{\text{ETD}} = 0, ~\forall ~\vec{k}_a,\vec{k_b} \in \mathbb{Z}^2.
\label{eq:E_ETD}
\end{equation}

Hence, the ETD combination of the three linear waves in the triadic propagator is exact at leading order and no triadic error is generated. This result is a consequence of using a Dahlquist analysis of the triadic frequency to derive the triadic error. This does not imply that ETD methods solve these equations exactly, as the triadic frequency is not the only source of error in the spectral evolution equation (\ref{eq:sigmaeqn}). Instead, it reflects that ETD schemes are effective at modelling linear waves, so this accuracy might carry over when viewing the nonlinear system as the result of interactions of three linear waves. A triadic error metric that considers additional components in the discretised error, such as in the interaction coefficients, would provide additional insight into the efficacy of ETD schemes in modelling nonlinear interactions.  \par

\subsection{Comparisons of selected methods}
\label{subsec:triadic_comps}
We now compare triadic errors in the discretised RSWEs for instances of the previously analysed timestepping schemes: RK3, RK4, trapezoidal, AB3 and TR-BDF2. ETD-RK2 is not included in these comparisons due to the zero triadic error of ETD schemes (\ref{eq:E_ETD}). With our chosen spatial discretisation and $\epsilon$ value, the linear stability limits for RK3, RK4, and AB3 are 0.00764, 0.0125, and 0.00318 respectively (oscillatory stability limits for these and other timesteppers can be found in \textcite{Durran}). We consider a range of timesteps up to $\Delta t = 0.05$ to include the performance of the implicit methods beyond the stability limits of these explicit methods. \par 
We first examine how the triadic error differs from an analysis of the linear dispersion errors in the three individual waves. Individual waves are referred to by their wavenumber and mode type as $\Psi=(k,l,\alpha_{\Psi})$, with $\alpha_{\Psi} \in \{ -,S,+ \}$ denoting the fast negative, slow, and fast positive modes respectively. Consider the case of a directly resonant fast-slow-fast (type \textit{ii}) interaction. For this to be directly resonant, the fast waves of $\Psi_a$ and $\Psi_c$ require an equal magnitude wavenumber \parencite{Embid_1996}. For example, consider $\Psi_a = (5,0,+), \Psi_c = (0,5,+),$ which are both fast waves with $K = 5, \omega = \sqrt{26}$. These waves have a different spatial orientation but exhibit the same linear dynamics and dispersion errors, shown in Figure \ref{fig:linear_vs_nonlinear_triadic_error_a}. The slow wave of $\Psi_b = (-5,5,S) $ has no triadic error, so is not shown. The triadic errors of this FSF interaction in Figure \ref{fig:linear_vs_nonlinear_triadic_error_b} are distinctly different for the timesteppers, compared to the linear errors. The trapezoidal method has linear dispersion errors for the fast waves, but these cancel in the triadic error metric. RK3 has a lower error than TR-BDF2 for the linear fast waves, yet has a higher triadic error for the FSF interaction. As the triadic error reflects the \textit{composition} of linear dispersion errors within the triadic propagator approximation, it is a different quantity to the separate linear dispersive errors.

\begin{figure}
     \centering
     \begin{subfigure}[b]{0.48\textwidth}
         \centering
         \includegraphics[scale=0.48]{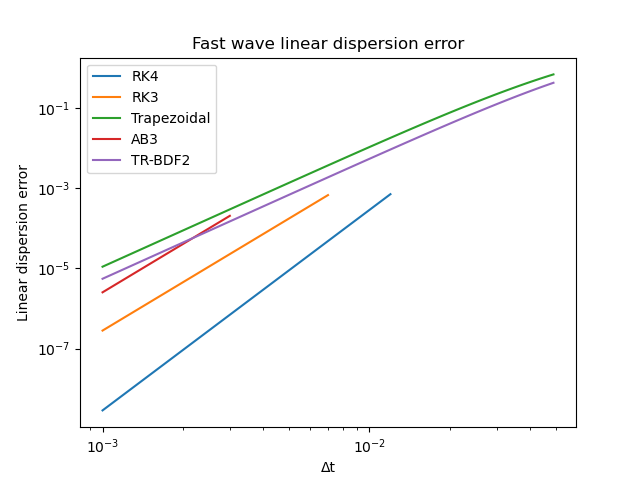}
         \caption{Fast linear wave errors}
         \label{fig:linear_vs_nonlinear_triadic_error_a}
     \end{subfigure}
     \hfill
     \begin{subfigure}[b]{0.48\textwidth}
         \centering
         \includegraphics[scale=0.48]{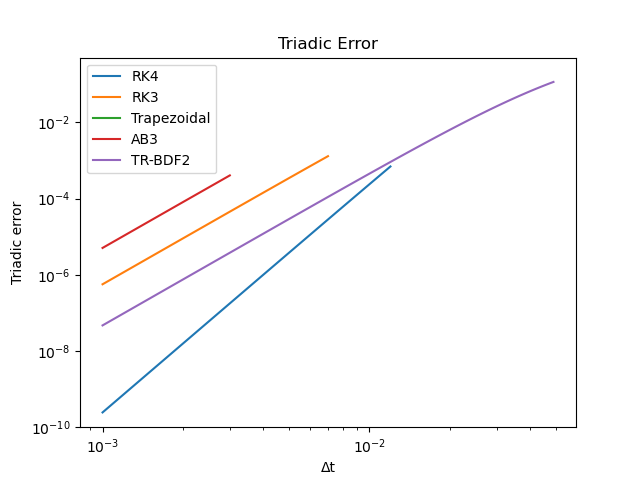}
         \caption{Triadic error }
         \label{fig:linear_vs_nonlinear_triadic_error_b}
     \end{subfigure}
        \caption{Contrasting the timestepping linear dispersion errors for a fast wave of $K = 5, \omega = \sqrt{26}$ in (a), with the triadic error of its interaction in a directly resonant fast-slow-fast (FSF) triad in (b). The trapezoidal method has linear dispersion errors for the fast waves, but these cancel in the triadic error metric for this direct resonance. RK3 has a lower error than TR-BDF2 for the linear waves, yet has a higher error for the complete triadic interaction.}
    \label{fig:linear_vs_nonlinear_triadic_error}
\end{figure}

We next consider the triadic errors of these selected schemes at the two cut-off frequencies, $\Omega_C$.
At $\Omega_C = 0.1$ (Figure \ref{fig:triadic_error_diff_methods_omegaC_0.1}), the trapezoidal method has zero triadic error, like ETD-RK2, due to the cancellation of linear fast wave errors in the triadic error measure. At the smallest timestep of $\Delta t =0.001$, there is a clear separation of errors for the remaining schemes; by order of reducing triadic error, we have AB3, RK3, TR-BDF2, then RK4. AB3, RK3, and TR-BDF2 have a very similar rate of error increase with $\Delta t$, although this slows down for TR-BDF2 beyond the explicit timestep limits. RK4 has a steeper increase in triadic error with step size and becomes less accurate than TR-BDF2 for $\Delta t \geq 0.0046$. \par
The relative accuracy of the methods change
when including a large number of near-resonant triads for the $\Omega_C = 5$ case (Figure \ref{fig:triadic_error_diff_methods_omegaC_5}).
The trapezoidal method now generates a considerable triadic error; it is less accurate than TR-BDF2 for the majority of considered timesteps and is only slightly more accurate for $\Delta t \geq 0.021$. AB3 and RK3 have a similar rate of error increase again, although this time they perform better relative to the implicit methods of trapezoidal and TR-BDF2. RK4 commits the least triadic error for the smallest timesteps and remains more accurate than TR-BDF2 until $\Delta t \geq 0.0076$, which is a larger range of timesteps than at $\Omega_C = 0.1$. \par
For both cut-off frequencies with this discretisation, AB3 has the largest triadic error, even when only considering its physical modes. The trapezoidal scheme has a smaller triadic error than TR-BDF2 for higher proportions of directly resonant triads. RK4 is very accurate for small timesteps, but has a rapid increase in triadic errors with $\Delta t$ as the truncation errors in its stability polynomial dominate; this means that it can be less accurate than implicit methods before reaching its stability limit. TR-BDF2 is unconditionally stable for any timestep, but the tradeoff for this stability is a large damping of the oscillatory waves \parencite{TR_BDF2_analysis} which is reflected in a significant triadic error at small timesteps.

 \begin{figure}
        \centering
        \begin{subfigure}[b]{0.475\textwidth}
            \centering
            \includegraphics[width=\textwidth]{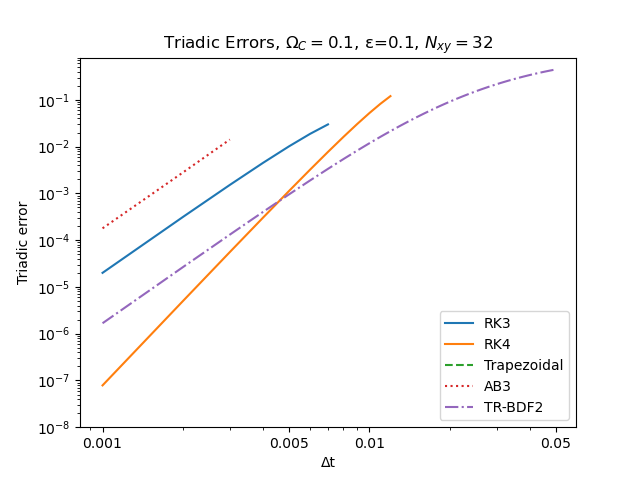}
            \caption{$\Omega_C = 0.1$}
            \label{fig:triadic_error_diff_methods_omegaC_0.1}
        \end{subfigure}
        \hfill
        \begin{subfigure}[b]{0.475\textwidth}  
            \centering 
            \includegraphics[width=\textwidth]{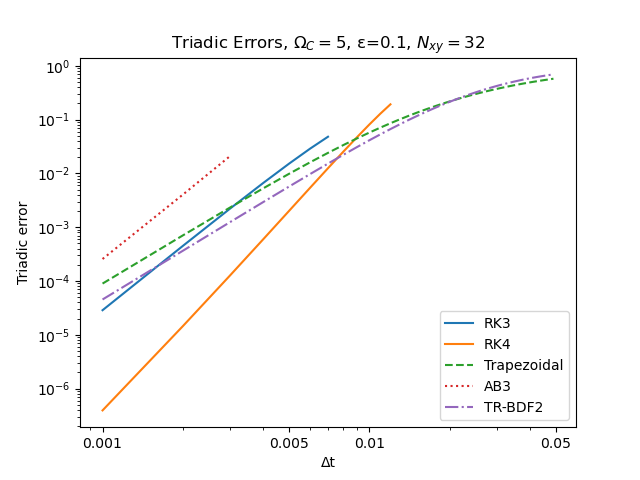}
            \caption{$\Omega_C = 5$}
            \label{fig:triadic_error_diff_methods_omegaC_5}
        \end{subfigure}
        \caption{Comparing triadic errors of RK3, RK4, trapezoidal, AB3, and TR-BDF2 for $ \Omega_C \in \{ 0.1,5 \}$. Oscillatory stability limits are included, which restrict the allowable timesteps of the AB3, RK3, and RK4 methods. For all the cases, AB3 has the largest errors. RK3 has a very similar rate of error increase with $\Delta t$, but is about an order of magnitude more accurate. RK4 has a much lower error for small step sizes but a much faster rate of error increase with $\Delta t$. As such, it is less accurate than TR-BDF2 at its largest timesteps, although it performs better for longer with $\Omega_C = 5$. The trapezoidal method has no error for direct resonances, so is not shown in (a), but has a larger error than TR-BDF2 for most timesteps with $\Omega_C = 5$.}
    \label{fig:triadic_error_diff_methods}
\end{figure}

\section{Numerical Models}
\label{section:numerical_models}
So far, we have described and used a new nonlinear error metric for timestepping schemes with known stability polynomials from the oscillatory Dahlquist test equation. However, this approach is not feasible for all numerical models, as the time discretisation can be highly integrated with physics and other model features, such as in LFRic. Hence, our second method of error quantification will provide test cases that can be applied to any numerical model. An additional difference from the triadic error is that we now consider errors in a solution over a time interval, as opposed to one timestep. \par
For this paper, we compare selected timesteppers within three different numerical models. The first model is a pseudospectral code \parencite{boyd2001chebyshev,canuto2007spectral}. The second is Gusto, a compatible finite element toolkit for solving atmospheric dynamical PDEs. The Gusto software is open source and can be accessed at \url{http://firedrakeproject.github.io/gusto}. The third model is LFRic, which is currently in development at the Met Office \parencite{Melvin_gungho}. \par
In these test cases, we seek to isolate the timestepping error from other sources of error in the models. To do this, we run simulations with a range of large timesteps and compare these to a small timestep reference solution, generated within that model. This reference solution is not exact, as all numerical methods introduce discretisation errors (e.g. the pseudospectral method has errors in the wave frequencies,  Fourier transform treatment of the nonlinearity, dealiasing, etc.). However, it is sufficient to examine the increase in timestepping error with larger $\Delta t$. To study errors generated within the nonlinear interactions, our comparison metric will measure discrepancies in spectral coefficients, $\sigma_{\vec{k}}^{\alpha}$, that appear in the spectral evolution ODE (\ref{eq:sigmaeqn}). \par
These numerical models all use dimensional versions of the RSWEs, like (\ref{eq:rswes_dim}). As such, we quote test case parameters with dimensions. We choose these to ensure a faster timescale for the linear oscillations than the nonlinear interactions, which is analogous to having a small $\epsilon$ in the standard equation (\ref{eq:standard-1}). We use $f=10 ~\text{s}^{-1}, ~g=50 ~\text{m} ~\text{s}^{-2}, ~H_0=2 ~\text{m}$, as this means that $Ro = Fr = \epsilon = 0.1$ when using characteristic scalings of $\mathtt{U} = \mathtt{L} = 1$. The nondimensional numbers obtained with our numerical models are typically smaller than 0.1 (Table \ref{table:test_case_nonD_nums}), which verifies that there is a separation of timescales.

\begin{table}
\caption{Quoting the Rossby and Froude numbers for the numerical models in both test cases. The finite element models of Gusto and LFRic use the same parameters, so have the same nondimensional numbers. The characteristic velocity of $\mathtt{U}$ is computed as the average over the simulation of the maximum RMS velocities. The characteristic length scale is set as $\mathtt{L} = L_{\vec{x}}$. The values of $Ro$ and $Fr$ in our experiments are less than or approximately equal to 0.1, so this is equivalent to $\epsilon \leq 0.1$ in the standard form (\ref{eq:standard-1}) RSWEs. This means that there is a separation of timescales between fast linear oscillations and slow nonlinear interactions.}
\begin{center}
\begin{tabular}{ |c|c|c|c| }
\hline
Test Case &
Numerical Model &
$Ro$ &
$Fr$ \\ 
\hline\hline
\hline
\multirow{2}{4em}{\centering 1} & Pseudospectral & 0.017 & 0.11 \\ 
 & Finite Element & 0.0012 & 0.012 \\
\hline
\multirow{2}{4em}{\centering 2a} & Pseudospectral & 0.0025 & 0.016 \\ 
& Finite Element & 0.0016 & 0.016 \\
\hline
\multirow{2}{4em}{\centering 2b} & Pseudospectral & 0.0073 & 0.046 \\
 & Finite Element & 0.0040 & 0.040 \\
\hline
\end{tabular}
\end{center}
\label{table:test_case_nonD_nums}
\end{table}

\subsection{Pseudospectral model}
The pseudospectral model transforms the variables between physical and spectral space with the use of the Fast Fourier Transform (and the potential to use 2/3 de-aliasing) to compute the multiplication and spatial differentiation operations. Five timestepping schemes from Section \ref{section:triad_errors} (RK4, trapezoidal, AB3, TR-BDF2, ETD-RK2) are compared in this model.

\subsection{Gusto}
Gusto is a toolkit providing stable, accurate timestepping schemes for compatible finite element discretisations of the partial differential equations that govern geophysical fluid flows. The compatible finite element framework ensures that the fundamental vector calculus identity of $\nabla \cdot \nabla \times = 0$ is preserved by the discretisation, leading to an accurate representation of geostrophic modes; this is crucial for accurate numerical weather prediction \parencite{cotter_shipton_mixed_FE}. Gusto provides a variety of different compatible finite element spaces. We will use the pair of finite element function spaces $\text{BDM2-P1}_{\text{DG}}$ for the velocity and fluid depth, $D$, on a mesh with triangular elements. \par
Gusto provides a range of different timestepping schemes. Here, we use two explicit timesteppers --- RK4 and the strong stability preserving Runge-Kutta method of SSPRK3 --- along with an implicit midpoint (ImpMid) timestepper. We also use an iterated semi-implicit timestepping scheme similar to that in LFRic \parencite{James_Kent_LFRic,Melvin_gungho}; we  refer to this approach as the Semi-Implicit Quasi-Newton (SIQN) scheme. SIQN uses two loops: an inner loop to solve the fast forcing terms implicitly and avoid explicit timestepping restrictions, and an outer loop for the slower transport terms. A difference between the use of SIQN in the two models is that Gusto allows for the use of different timesteppers for transporting each field. This can be required for stability, as the order of the finite element function spaces is higher in Gusto than in LFRic. At the time of these results, Gusto applied the SIQN method with four iterations in the outer loop and one in the inner loop. For more discussion of SIQN in Gusto and pseudocode, see \textcite{hartney2024compatible}. \par

\subsection{LFRic}
LFRic is the next generation weather and climate model in development at the Met Office; a detailed overview of the shallow water component of the model is found in \textcite{James_Kent_LFRic}. Its newly designed dynamical core, GungHo \parencite{Melvin_gungho, melvin2024mixed}, uses a compatible finite element discretisation on lowest-order quadrilateral elements. The LFRic shallow water mini app is used for the results in this paper. This solves for prognostic variables of a velocity and geopotential of $\phi_L = gh$, using the mixed $\text{RT1-Q0}$ function space. LFRic's default timestepper is the SIQN approach, with two inner loops and two outer loops. Two explicit timesteppers, SSPRK3 and RK4, will also be examined.

\subsection{A spectral coefficient error metric}
We use an error metric that quantifies differences in the spectral coefficients of a timestepped solution. This approach elucidates inaccuracies in nonlinear interactions, as these are reflected in the energy distribution over the wavenumbers and modes. \par
We consider a Fourier space representation of a solution to the RSWEs, similar to (\ref{eq:spec_v}) but without the $\exp(i(\vec{k} \cdot \vec{x}))$ component,
\begin{equation}
    \vec{\hat{U}}(\vec{k},t) =  \sum_{\alpha \in \{ -1, 0, +1 \} } \sigma^{\alpha}_{\vec{k}} (t) ~\vec{r}^{\alpha}_{\vec{k}},
    \label{eq:spec_u}
\end{equation}

\noindent where $\hat{\vec{U}}(\vec{k}, t) = \mathcal{F} \{ \vec{U}(\vec{x}, t) \}$ denotes the spatially Fourier transformed solution. When using the symmetrised equations (\ref{eq:rswes_skew}) with a skew-Hermitian $\mathcal{L}$, the spectral amplitude of each mode can be computed using the approach in \textcite{Ward_Dewar}: 
\begin{equation}
    \label{eq:spectral_amp}
    \sigma^{\alpha}_{\vec{k}}(t) = \left( {\vec{r}^{\alpha}_{\vec{k}}} \right)^{*} \cdot \vec{\hat{U}}(t) .
\end{equation}

\noindent We are interested in the difference between spectral coefficients in the large timestepped solution, $\sigma_{\Delta t, \vec{k}}^{\alpha}$, relative to the reference solution generated with a small timestep, $\sigma_{\vec{k}}^{\alpha}$, at a specified time of $t_n$:
\begin{equation}
    \Delta \sigma_{\vec{k}}^{\alpha} (t_n)
    =
    \left\Vert \sigma_{\vec{k}}^{\alpha} (t_n) - \sigma_{\Delta t, \vec{k}}^{\alpha} (t_n) \right\Vert.
    \label{eq:spec_diff}
\end{equation}

\noindent Visualisations of $\Delta \sigma_{\vec{k}}^{\alpha}$ (Figures \ref{fig:test_case_2b_waveno_errors}, \ref{fig:gusto_waveno_errors}, \ref{fig:lfric_waveno_errors}) allow observations of which nonlinear interactions a timestepper is misrepresenting. In doing so, we can examine the accuracy in modelling specific direct and near-resonant triads. \par
We also compute a total spectral coefficient error metric for an easier comparison of many timesteppers and timestep sizes in a given model (Figures \ref{fig:Gaussian_pseudo_results}--\ref{fig:Gaussian_LFRic}, \ref{fig:Test_case_2_pseudo}, \ref{fig:Test_case_2_gusto}, \ref{fig:Test_case_2_lfric}). We sum the spectral differences over all wavenumbers and mode types, at $N_T$ equispaced sample time locations $t_n$, to find
\begin{equation}
    SE 
    =
    \frac{1}{N_T} \sum_{t_n} 
  \sum_{\vec{k} \in \mathbb{Z}^2} \sum_{\alpha \in \{-1,0,+1\} } \Delta \sigma_{\vec{k}}^{\alpha} (t_n).
    \label{eq:spec_error_metric}
\end{equation}

For the results in this paper, the reference solution for each numerical model is computed with RK4 and $\Delta t = 10^{-4} ~\text{s}$. Spectral errors for the pseudospectral method will only be computed within the region unaffected by aliasing ($|k|< (2/3) ~k_{\text{max}}, |l| < (2/3) ~l_{\text{max}}$). \par
As the spectral coefficient expression ($\ref{eq:spec_u}$) uses the eigenbasis for the symmetrical geopotential ($\vec{r}^{\alpha}_{\vec{k}}$ expressions of (\ref{eq:eigenvecs}), (\ref{eq:eigenvecs_K0}) in Appendix A), other elevation measures i.e. the depth field ($D$) from Gusto and geopotential height ($\phi_L$) from LFRic, first need to be converted to $\phi$. For Gusto, this is
\begin{equation}
    \label{eq:gusto_geopot_conversion}
    \phi = \sqrt{\frac{g}{H_0}} (D - H_0),  
\end{equation}

\noindent and for LFRic
\begin{equation}
    \label{eq:lfric_geopot_conversion}
    \phi = \sqrt{\frac{g}{H_0}} \left( \frac{\phi_L}{g} - H_0 \right) .
\end{equation}

\section{Test case 1: Gaussian initial conditions}
\label{section:gaussian_test_case}
The first test case considers a Gaussian height perturbation at a state of rest. Similarly to the one-dimensional example of Section \ref{subsec:RSWEs_effect_of_eps}, the dynamics of this Gaussian shape dispersing and reforming continues over the simulation. The time between consecutive reformations is primarily determined by the speed of the linear waves, as seen in Figure \ref{fig:effectofepsongaussian}. However, the nonlinear terms introduce a slow phase shift, analogous to that in Figure \ref{fig:Gauss_lin_nonlin}, which changes the frequency of these reformations. 

\subsection{Specifications}
We define a Gaussian height perturbation of peak value $\eta_G$ above a base height of $H_G$. This is located at the centre of a square domain of size $[0,L_{\vec{x}}] \times [0,L_{\vec{x}}]$, with periodic boundaries in both spatial dimensions. The decay rate of the perturbation is controlled by $\sigma_r$. The initial velocities are identically zero:
\begin{equation}
    u(0,x,y) = 0, ~v(0,x,y) = 0, ~h(0,x,y) = H_G + \eta_G  \exp \left( \frac{ -\left( x-\frac{L_{\vec{x}}}{2} \right)^2 - \left( y-\frac{L_{\vec{x}}}{2} \right)^2 }{\sigma_r} \right) .
\end{equation}

We apply this test case with $N_x = N_y = 32$ grid points or elements and $H_G = 2 ~\text{m}$. The pseudospectral model uses $L_{\vec{x}} =2 \pi ~\text{m}, \eta_G = 1 ~\text{m}, \sigma_r = 0.656 ~\text{m}^2$, and the finite element methods (Gusto, LFRic) use $L_{\vec{x}}=10 ~\text{m}, \eta_G=0.1 ~\text{m}, \sigma_r=1.66 ~\text{m}^2$. The pseudospectral model is run until $T_{\text{end}}=100 ~\text{s}$ and the finite element methods use $T_{\text{end}}=50 ~\text{s}$. Spectral errors are computed in each model using solution samples at every $\Delta T = 0.05 ~\text{s}$, with the total error metric being the mean of these evaluations (\ref{eq:spec_error_metric}).  \par
The pseudospectral model uses a mild hyperviscosity of order $(\nabla^2)^4$, with coefficient $\mu = 4 \times 10^{-12} ~\text{m}^8 ~\text{s}^{-1}$. The allowable timesteps in Gusto and LFRic are restricted relative to Courant numbers for the advective ($Cr_{\vec{u}}$) and gravity wave ($Cr_{c}$) speeds:
\begin{subequations}
\begin{align}
    Cr_{\vec{u}} &= \max_{u,v} \left( \frac{u \Delta t}{\Delta x} + \frac{v \Delta t}{\Delta y} \right), \\ 
    Cr_{c} &= \frac{c \Delta t}{\Delta x} = \frac{\Delta t}{\epsilon \Delta x} .
\end{align}
\label{eq:Cr_numbers}
\end{subequations}

\noindent For this test, the gravity wave Courant number scales with the timestep, as ${Cr_{c} = 32 \Delta t}$. As the advective Courant number is at least an order of magnitude smaller for all the simulations, only $Cr_{c}$ restricts the allowable timestep size. A range of larger timesteps are considered from $\Delta t = 5 \times 10^{-4} ~\text{s}$ to $ 5 \times 10^{-2} ~\text{s}$, which have respective $Cr_{c}$ of $0.016$ to $1.6$. \par

\subsection{Pseudospectral results}
The total spectral errors for the pseudospectral method are shown in Figure \ref{fig:Gaussian_pseudo_results}. AB3 has similarly large errors to the two implicit schemes of trapezoidal and TR-BDF2. These implicit methods are stable at the largest timestep but have significantly higher spectral errors than RK4 and ETD-RK2. TR-BDF2 performs better at smaller timesteps than the trapezoidal method, with the latter having a slightly lower error for the largest two timesteps. RK4 has the lowest error at the smallest three timesteps, which is likely a result of its higher order than ETD-RK2. There is a steeper increase in the RK4 error, such that ETD-RK2 performs slightly better at $\Delta t = 0.01 ~\text{s}$.  \par

\begin{figure}
    \centering
     \includegraphics[scale = 0.7]{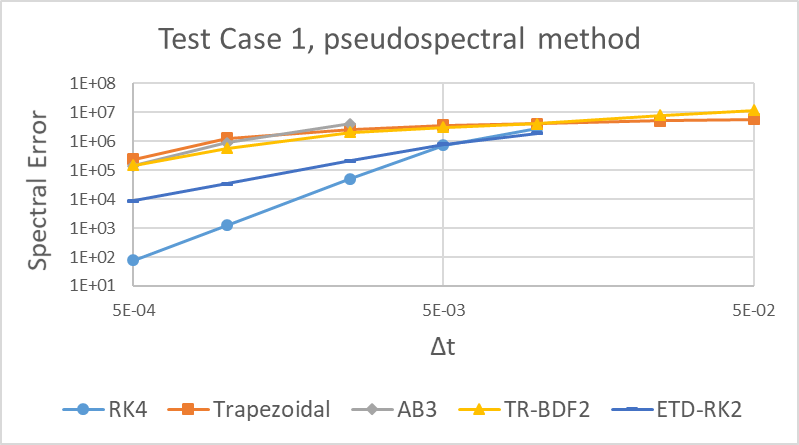}
    \caption{Spectral coefficient errors for the pseudospectral method in test case 1. RK4 performs best for the smallest timesteps, but ETD-RK2 overtakes this at the last explicit stable time step, due to a slower growth of errors with increasing $\Delta t$. TR-BDF2 has a lower error than the trapezoidal method for the smaller timesteps, but the latter is slightly better at the largest two timesteps.}
    \label{fig:Gaussian_pseudo_results}
\end{figure}

\subsection{Gusto results}
The two implicit approaches of SIQN and implicit midpoint are stable for all the considered step sizes, with the trade-off being a considerably larger error than the explicit methods of RK4 and SSPRK3 (Figure \ref{fig:Gaussian_gusto}). The implicit midpoint method has a lower error at the largest timestep, but for all other $\Delta t$ values, using different time discretisations for the transport of each field within the SIQN scheme produces a minimal difference in errors. RK4 outperforms SSPRK3, which is in line with its higher-order accuracy.  \par

\begin{figure}
    \centering
     \includegraphics[scale = 0.7]{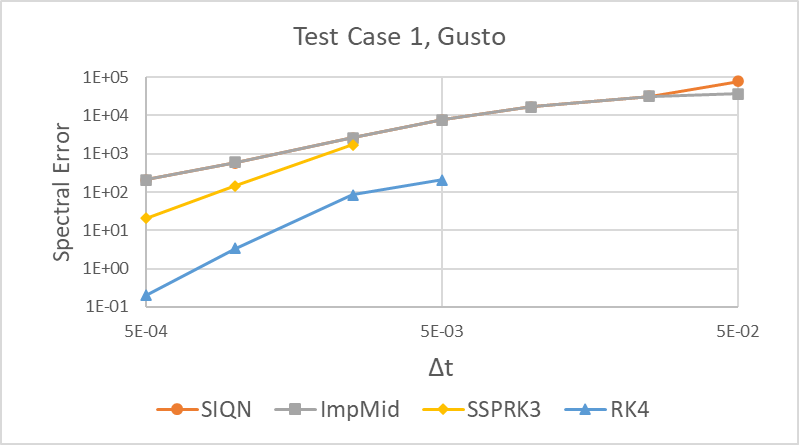}
    \caption{Spectral coefficient errors for Gusto in test case 1. The implicit methods (SIQN, implicit midpoint) have a much higher error than the explicit methods and perform very similarly, except at the largest timestep. RK4 has much lower errors than SSPRK3 for all its resolvable timesteps.}
    \label{fig:Gaussian_gusto}
\end{figure}

\subsection{LFRic results}
The semi-implicit (SIQN) method has a reasonably constant increase in spectral coefficient error with timestep on the logarithmic axes of Figure \ref{fig:Gaussian_LFRic}. So does SSPRK3, which has a much lower error at the smallest $\Delta t$, but a faster rate of error increase with timestep. This means that the SIQN and SSPRK3 errors are very similar at $\Delta t =0.01 ~\text{s}$. RK4 has very small errors for the three smallest timesteps, before a very rapid increase in errors at the next two largest timesteps. The relative performance of the schemes is consistent with their orders --- SIQN is second-order, SSPRK3 is third-order, and RK4 is fourth-order.\par

\begin{figure}
    \centering
     \includegraphics[scale = 0.7]{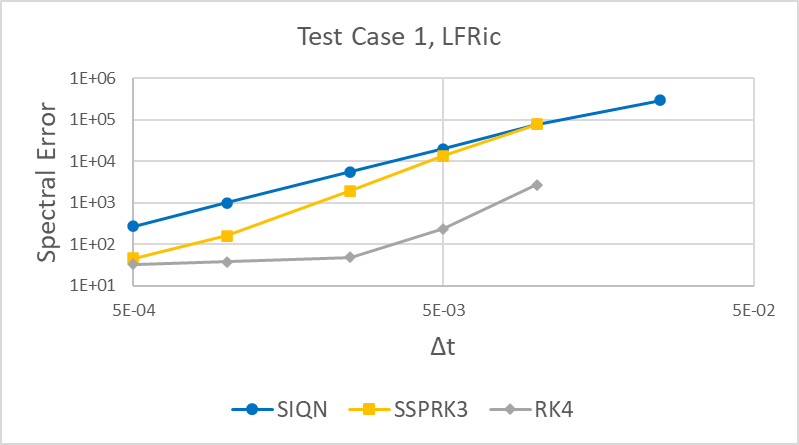}
    \caption{Spectral coefficient errors for LFRic in test case 1. SIQN and SSPRK3 both have relatively consistent increases in error. SSPRK3 has a lower error for smaller $\Delta t$ but a faster rate of error increase than SIQN. RK4 is the most accurate until its stability limit, but has a rapid upturn in error for its last two timesteps.}
    \label{fig:Gaussian_LFRic}
\end{figure}

\section{Test case 2: Triadic initial conditions}
\label{section:triad_test_case}

We now move to the second type of test case, where individual linear waves are selected to excite specific triadic interactions. This examines timestepping accuracy in the presence of low-frequency dynamics resulting from direct- or near-resonant triads. Linear waves of a specific wavenumber and mode are constructed using the RSWEs eigenmode basis (Appendix A). This approach has enabled studies of triadic energy exchanges at different $\Omega$ \parencite{Alex_Owen_thesis} and the interplay of gravity waves with potential vorticity \parencite{Ward_Dewar}. The ability to trigger specific dynamical features by choosing the initial waves makes this a versatile test. We illustrate this with two different superpositions of linear waves in test cases 2a and 2b. Both scenarios contain a slow transition from linear to nonlinear dynamics as the initial waves interact nonlinearly through triads. These interactions occur between different combinations of mode types (recall Section \ref{subsec:rswe_triad_interactions} and Table \ref{table:triad_list}). \par

\subsection{Specifications}
A linear wave of wavenumber $\vec{k}$ and mode $\alpha$ can be constructed by projecting the corresponding eigenfunction onto physical space, $\vec{m}^{\alpha}_{\vec{k}} = Re \{ \exp(i(k x+l y)) ~\vec{r}_{\vec{k}}^{\alpha} \}$. The doubly periodic $f$-plane expressions for $\vec{m}^{\alpha}_{\vec{k}}$ are explicitly stated in Appendix A. For this test case, the wavenumbers are scaled relative to the domain length, as $k,l = 2 \pi z/ L_{\vec{x}}, ~z \in \mathbb{Z}$. This ensures the same number of periods are contained in the square domain irrespective of its size, e.g., a $\vec{k}=(5,0)$ wave contains five periods in the $x$-dimension for any value of $L_{\vec{x}}$. Chosen waves are given a nonzero initial amplitude, $\sigma^{\alpha}_{\vec{k}} (0)$, with the majority of the spectral coefficients remaining as zero. The initial condition is the superposition of the selected fast and slow waves, similarly to (\ref{eq:spec_v}),
\begin{equation}
    \vec{U}(x,y,0) = \sum_{\vec{k} \in \mathbb{Z}^2} \sum_{\alpha \in \{ -1, 0, +1 \} } \sigma^{\alpha}_{\vec{k}} (0) ~\vec{m}^{\alpha}_{\vec{k}} .
    \label{eq:triad_test_case_IC}
\end{equation}

\noindent As the $\vec{m}^{\alpha}_{\vec{k}}$ expressions are derived for the symmetrical geopotential, a conversion to the model-specific elevation prognostic variable is required, i.e. using relationship (\ref{eq:gusto_geopot_conversion}) for Gusto and (\ref{eq:lfric_geopot_conversion}) for LFRic.  \par
The two versions of the triadic test case are run with a spatial resolution of $N_x = N_y = 64$ (grid points or elements), using the same domain sizes ($L_{\vec{x}}$) for each model as in test case 1. A total simulation time of $T_{\text{max}}=50 ~\text{s}$ is used, with spectral coefficient errors computed using (\ref{eq:spec_error_metric}) from samples every $\Delta T = 0.1 ~\text{s}$. Initial amplitudes of $\sigma^{\alpha}_{\vec{k}}(0) = 0.1 $ are used. Larger magnitudes instigate more triadic energy exchange but make it more challenging numerically, i.e., requiring smaller timesteps. \par
For the pseudospectral model, a hyperviscosity coefficient of $\mu = 10^{-10} ~\text{m}^8 ~\text{s}^{-1}$ is used, along with a 2/3 de-aliasing procedure. For the finite element methods, the finer 64-by-64 spatial resolution means that there are stricter gravity wave Courant numbers than in the first test case: $Cr_{c} = 64 \Delta t$. With the same timestep range of $\Delta t = 5 \times 10^{-4} ~\text{s}$ to $ 5 \times 10^{-2} ~\text{s}$, the respective $Cr_{c}$ are $0.032$ to $3.2$. With these larger Courant numbers, the explicit timesteppers in both finite element models are now unstable at their largest stable timestep for test case 1. \par

\subsection{Test case 2a}
Case 2a initialises one fast wave of $\Psi_1 = (5,0,+)$ and one slow wave of $\Psi_2 = (-5,5,S)$. This is similar to a simulation in \textcite{Ward_Dewar} to analyse the effect of a slow mode on the redistribution of fast inertia-gravity wave energy. \par
The initial profile being a superposition of only two linear waves can be visualised in physical space (Figure \ref{fig:test_case_2a_IC}). Each field ($u, v, h$) contains five local maximums in the $x$ and $y$ dimensions. In the full RSWEs, the shapes and magnitudes of these maximums change during the simulation, whereas in a linear version, the fields move but the initial condition profiles are preserved. This is shown for the height field in Figure \ref{fig:test_case_2a_lin_nonlin}. \par

\begin{figure}
    \centering
     \includegraphics[scale = 0.48]{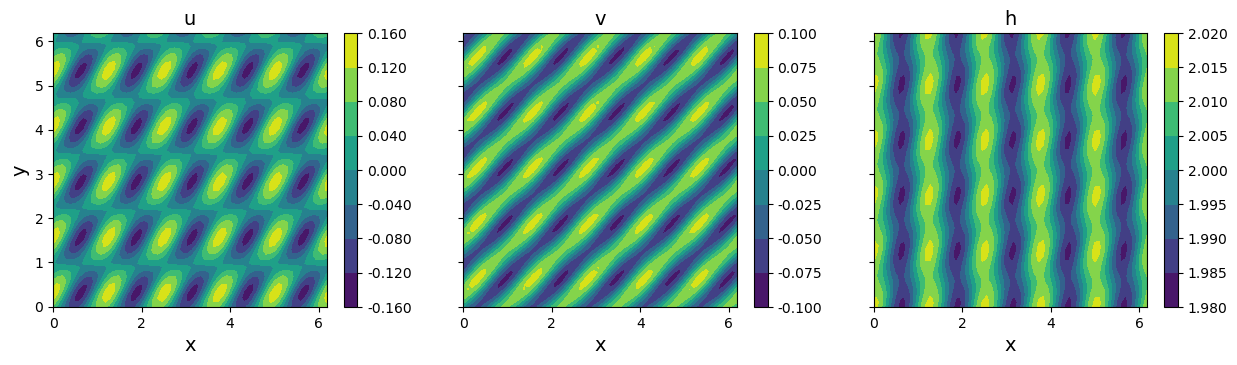}
     \caption[Test case 2a initial condition]{Test case 2a initial condition. These visualisations are of the fields obtained in the pseudospectral model.}
    \label{fig:test_case_2a_IC}
\end{figure}

\begin{figure}[htbp]
     \centering
     \begin{subfigure}{0.49\textwidth}
         \centering
         \includegraphics[width=\textwidth]{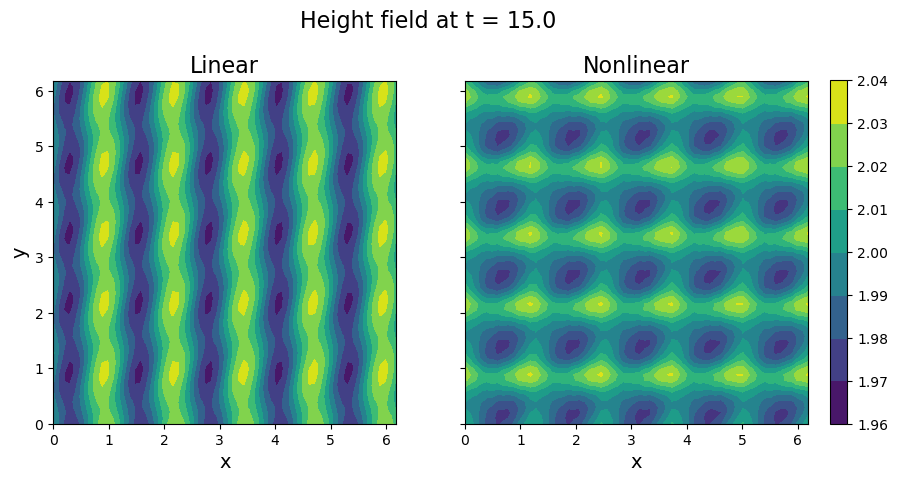}
         \caption{$t = 15 ~\text{s}$}
     \end{subfigure}
     \hfill
     \begin{subfigure}{0.49\textwidth}
         \centering
         \includegraphics[width=\textwidth]{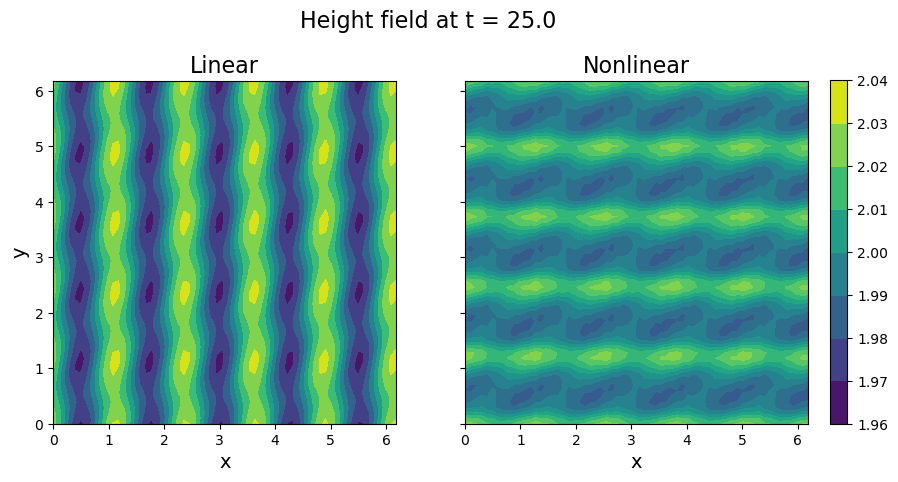}
         \caption{$t = 25 ~\text{s}$}
    \end{subfigure}
    \caption{The height field for test case 2a, at $t=15 ~\text{s}$ (a) and $t=25 ~\text{s}$ (b), for linear (left) and nonlinear (right) versions of the RSWEs. The height field in the linear case looks similar at both times, whereas there is a change in the shape and magnitude of the maximums in the nonlinear system.}
    \label{fig:test_case_2a_lin_nonlin}
\end{figure}

The presence of triadic interactions in this test case is highlighted by plotting spectral amplitudes, $\sigma^{\alpha}_{\vec{k}}$, in wavenumber space. Figure \ref{fig:OFOS_waveno_space} shows the distribution of spectral energy in the three modes at the start and end times of the simulation. The initial state in Figure \ref{sub_fig:OFOS_waveno_space_start} only contains initial energy at values of ($\vec{k}, \alpha$) corresponding to $\Psi_1$ and $\Psi_2$. The first noticeable triadic interaction is a self-excitation of the fast wave into a secondary fast wave at twice the wavenumber, $\Psi_1 + \Psi_1 = \Psi_4 = (10,0,+)$. This is a near-resonant FFF (type \textit{iii}) interaction, which reinforces that not only direct resonances are of importance to the computation of triadic errors. After a slightly longer time, we observe the expected FSF (type \textit{ii}) direct resonance with a third (fast) wave, $\Psi_1 + \Psi_2 = \Psi_3 = (0,5,+)$. At $T_{\text{max}}$, the majority of the energy in $\Psi_1$ has moved to $\Psi_3$. SFS (type iv) interactions at larger triadic frequencies also play a role in shifting the energy in the slow modes, e.g., $\Psi_2 + \Psi_4 = \Psi_5 = (5,5,S)$. This results in many more slow waves in the final state in Figure \ref{sub_fig:OFOS_waveno_space_final}.
As the initial wavenumbers are all integer multiples of five, the only triadic interactions in an accurate simulation involve wavenumbers that adhere to this constraint.

\begin{figure}[htpb]
     \centering
     \begin{subfigure}[t]{\textwidth}
         \centering
         \includegraphics[width=0.9\textwidth]{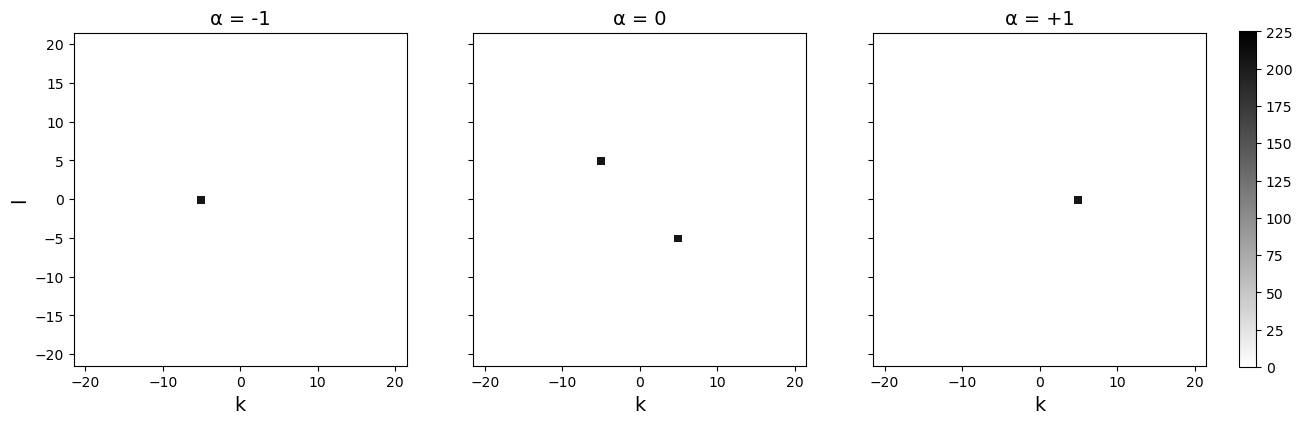}
         \caption{Spectral amplitudes in the initial state}
         \label{sub_fig:OFOS_waveno_space_start}
     \end{subfigure}
     \vskip\baselineskip
     \begin{subfigure}[t]{\textwidth}
         \centering
         \includegraphics[width=0.9\textwidth]{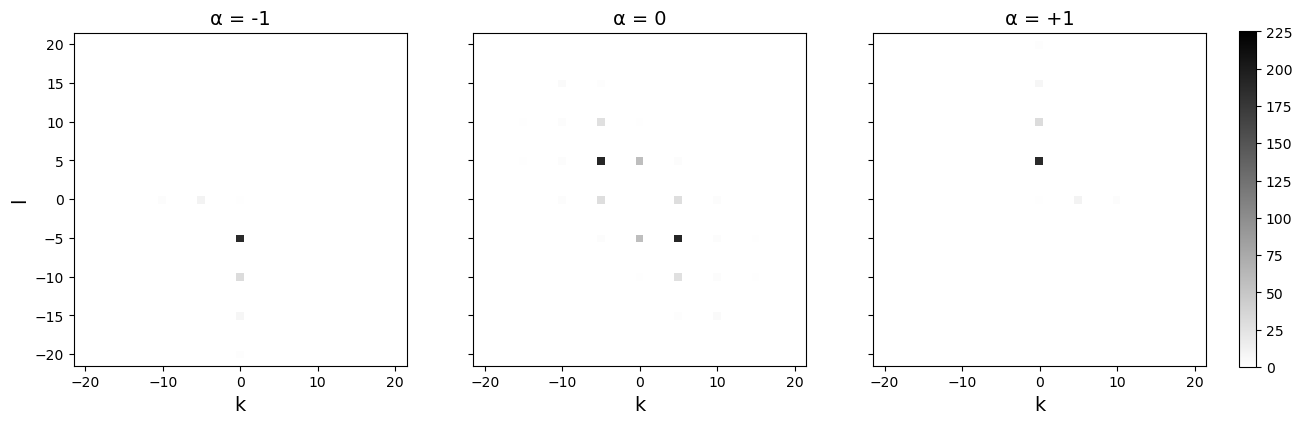}
         \caption{Spectral amplitudes in the final state}
         \label{sub_fig:OFOS_waveno_space_final}
     \end{subfigure}
        \caption[Initial and final spectral amplitudes in test case 2a]{Magnitudes of the spectral coefficients in spectral space for each mode, in test case 2a, at the initial (a) and final (b) simulation times. These spectra are computed from the reference pseudospectral solution. The three panels in each figure, from left to right, show the fast negative, slow, and fast positive modes. Nonzero spectral amplitudes are in greyscale, with black indicating the largest spectral amplitude. The initial state only contains energy for the selected two linear waves. However, four ($\vec{k}, \alpha$) combinations contain initial energy; this is because the spectral coefficients obey Hermitian symmetry of $\left( \sigma_{\vec{k}}^{\alpha} \right)^* = \sigma_{-\vec{k}}^{-\alpha}$ to ensure that the prognostic variables are real-valued. At $T_{\text{end}}$, the majority of fast wave energy has been exchanged in a directly resonant triadic interaction, from $\Psi_1$ to $\Psi_3$.  A larger number of slow modes in the final state contain energy due to higher triadic frequency interactions.}
    \label{fig:OFOS_waveno_space}
\end{figure}

\subsection{Test case 2b}
The second triadic test case contains two fast waves and a cluster of nine, low wavenumber, slow modes. The fast waves of $\Psi_1$ is retained from case 2a, with the other having an equivalent total wavenumber, $\Psi_6 = (-3,4,+)$. A larger number of slow waves, from the different combinations of $k,l \in \{ -1,0,1 \}$, leads to the initial state in physical space in Figure \ref{fig:test_case_2b_IC}. The addition of many low wavenumber slow waves has previously been used to investigate the role of near-resonances on energy redistribution \parencite{Alex_Owen_thesis}. \par

\begin{figure}
    \centering
     \includegraphics[scale = 0.48]{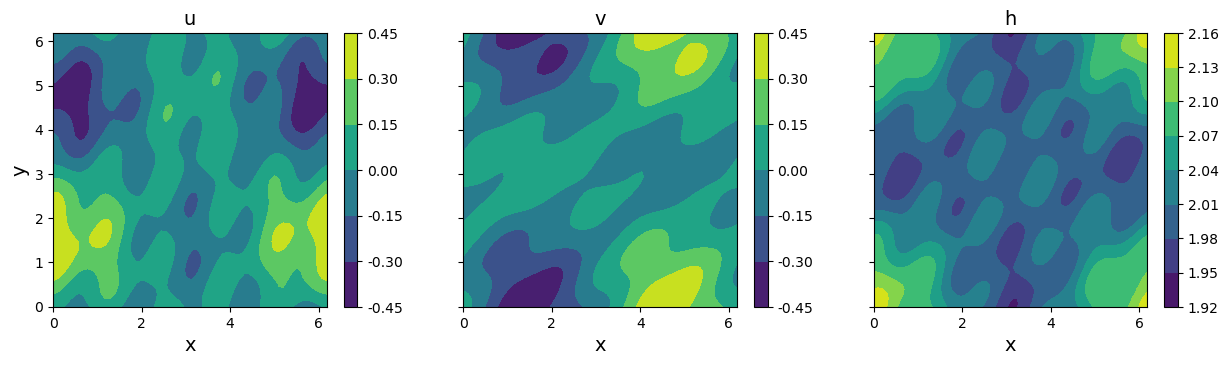}
    \caption[Test case 2b initial condition]{Test case 2b initial condition. These visualisations are of the fields obtained in the pseudospectral model.}
    \label{fig:test_case_2b_IC}
\end{figure}

The presence of more slow mode waves allows for a wider redistribution of the fast wave energy. Due to the smaller wavenumbers of the slow modes, directly resonant triads cannot form in the initial state; this energy exchange is purely due to near-resonances. In wavenumber space, this causes a ring of energy to appear in the fast modes at a locus of $K=5$ (Figure \ref{sub_fig:TwoFSCent_waveno_space_final}). Secondary rings also form at integer multiples of five, i.e. $K=10, 15, 20$, due to energy redistribution at wavenumbers excited by self-resonances, such as $\Psi_4$. \par
Although initialising a single fast wave is sufficient to observe this ring formation, using two allows for this energy redistribution to be observed after shorter simulation times. Larger initial amplitudes also allow for larger energy transfers and more distinct rings. However, this increases the numerical difficulty of the simulation; for example, a pseudospectral model requires a larger hyperviscosity to compensate.

\begin{figure}
     \centering
     \begin{subfigure}[t]{\textwidth}
         \centering
         \includegraphics[width=0.9\textwidth]{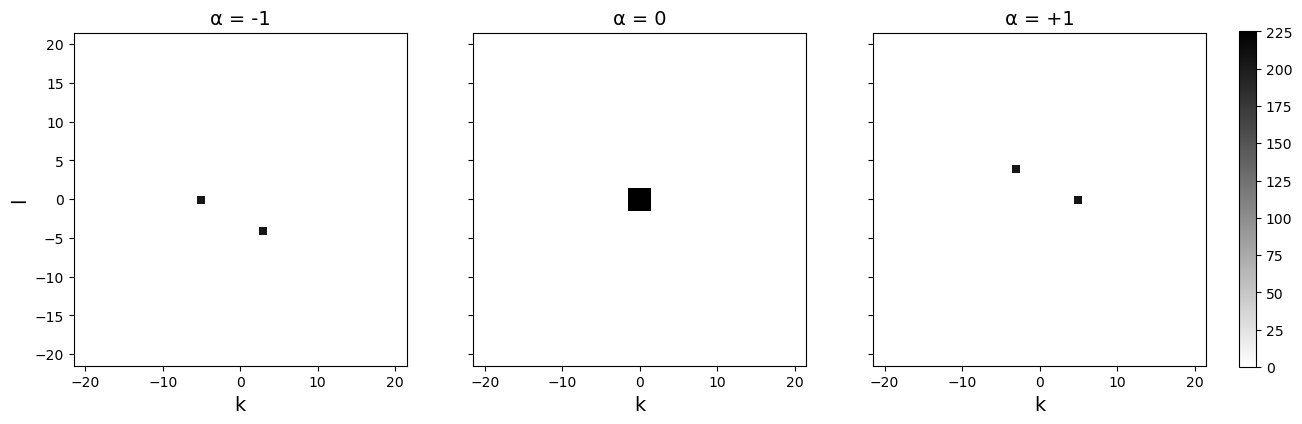}
         \caption{Spectral amplitudes in the initial state}
         \label{sub_fig:TwoFSCent_waveno_space_start}
     \end{subfigure}
     \vskip\baselineskip
     \begin{subfigure}[t]{\textwidth}
         \centering
         \includegraphics[width=0.9\textwidth]{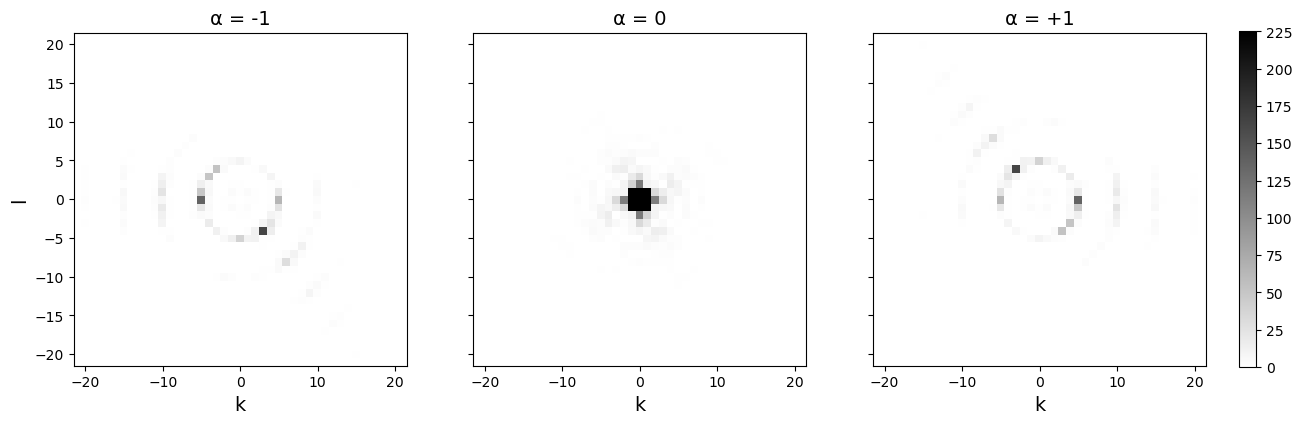}
         \caption{Spectral amplitudes in the final state}
         \label{sub_fig:TwoFSCent_waveno_space_final}
     \end{subfigure}
        \caption[Initial and final spectral amplitudes in test case 2b]{Magnitudes of the spectral coefficients in spectral space for each mode, in test case 2b, at the initial (a) and final (b) simulation times. These spectra are computed from the reference pseudospectral solution. The initial state in (a) contains two fast waves and nine slow ones. This allows for the redistribution of fast mode energy into rings in (b), due to near-resonant triadic interactions.}
    \label{fig:TwoFSCent_waveno_space}
\end{figure}

\subsection{Pseudospectral results}
In both cases 2a and 2b, the explicit schemes of RK4 and ETD-RK2 generate the lowest spectral coefficient error (Figure \ref{fig:Test_case_2_pseudo}). RK4 has a lower error for the smallest two timesteps but has a larger error increase with $\Delta t$ such that ETD-RK2 has the smaller error at the largest stable explicit timestep of $\Delta t = 0.005 ~\text{s}$. ETD-RK2 has a better performance relative to the other schemes in test case 2a, which uses fewer initial waves. AB3 has a similar amount of spectral error to the trapezoidal scheme. For both cases, TR-BDF2 has a smaller spectral error than the trapezoidal method at the finest three timesteps. The trapezoidal scheme performs better than TR-BDF2 for the largest timesteps; this is by a larger degree in case 2a than in 2b. \par
In Figure \ref{fig:test_case_2b_waveno_errors} we present visualisations of fast mode spectral differences ($\Delta \sigma_{\vec{k}}^{\alpha=-1} + \Delta \sigma_{\vec{k}}^{\alpha=+1} $) using (\ref{eq:spec_diff}) for test case 2b. These are computed at $T_{\text{max}}$ from solutions at a common timestep of $\Delta t =0.005 ~\text{s}$. AB3 is not shown, as this stepsize is beyond its stability limit. The errors are mainly constrained to the primary (inner) and secondary (second, third, fourth) fast wavenumber rings. For RK4, the largest errors are in the second ring, whilst for ETD-RK2 they are in the primary ring. RK4 and TR-BDF2 also have more noticeable errors around the self-resonance of the (-3,4,+) wave into the secondary rings, which is around the line of $l = -4k/3$. 

\begin{figure}
     \centering
     \begin{subfigure}[t]{0.48\textwidth}
         \centering
         \includegraphics[scale=0.48]{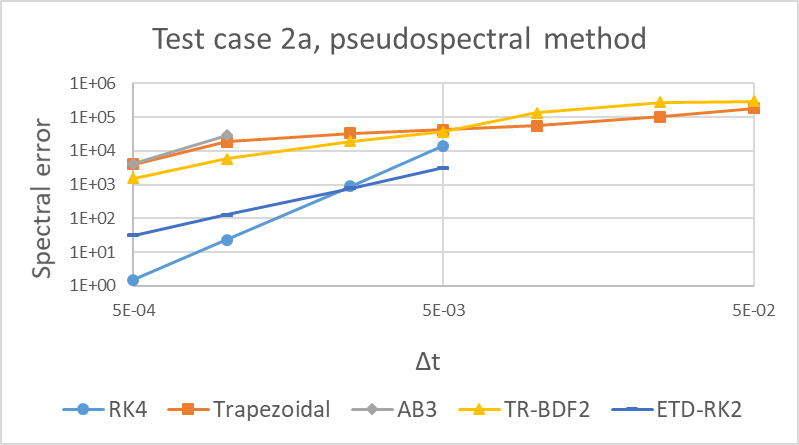}
         \caption{Test Case 2a}
     \end{subfigure}
     \hfill
     \begin{subfigure}[t]{0.48\textwidth}
         \centering
         \includegraphics[scale=0.48]{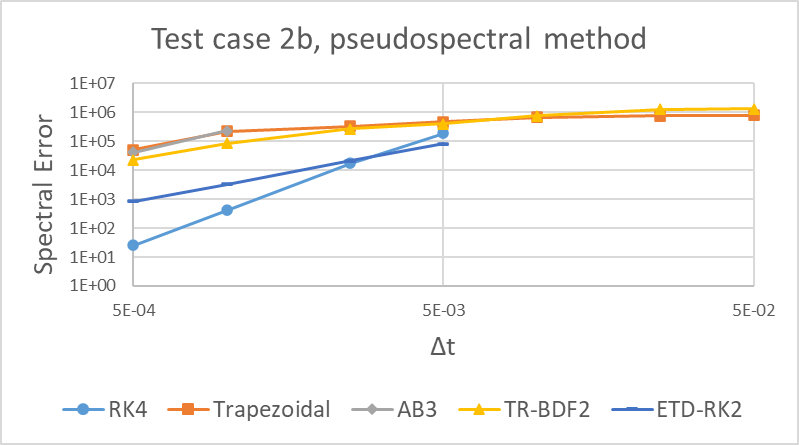}
         \caption{Test Case 2b}
     \end{subfigure}
        \caption{Spectral coefficient errors in test cases 2a and 2b for the pseudospectral method. Like with the Gaussian case, RK4 has the lowest error at the two smallest $\Delta t$, but ETD-RK2 has the lower error at the largest stable explicit timestep. The trapezoidal method is more accurate than TR-BDF2 at the largest timesteps for case 2a, which has a higher proportion of directly resonant triadic interactions.}
    \label{fig:Test_case_2_pseudo}
\end{figure}

\begin{figure}
    \centering
     \includegraphics[scale = 0.5]{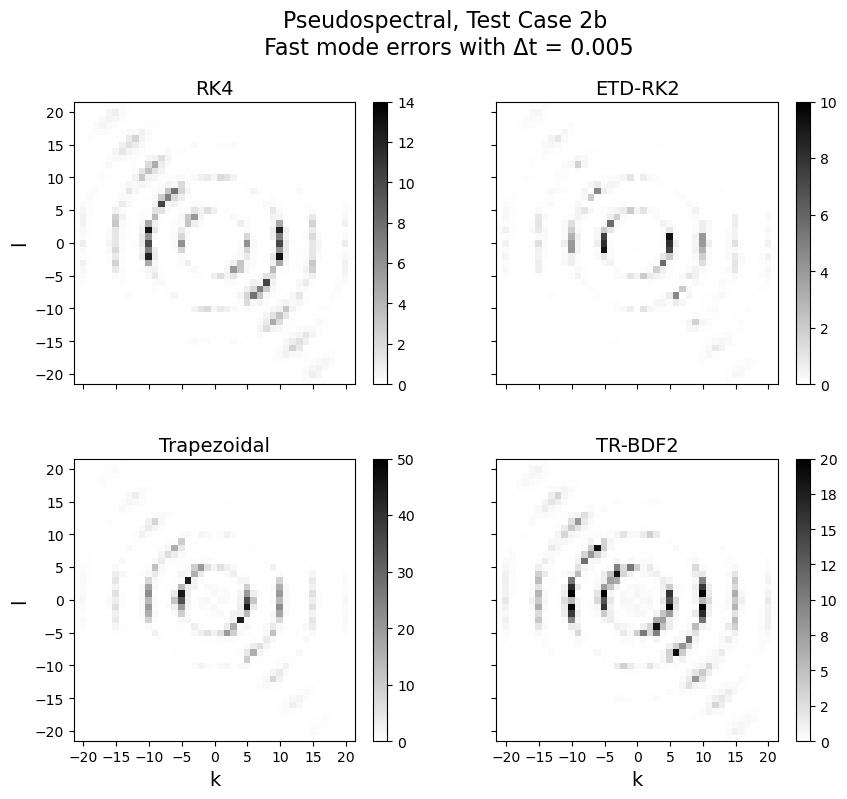}
    \caption[]{Spectral coefficient errors in the fast modes ($\Delta \sigma_{\vec{k}}^{\alpha=-1} + \Delta \sigma_{\vec{k}}^{\alpha=+1} $) for pseudospectral timesteppers at the end time of case 2b. The visible energy differences occur in the primary and secondary rings. The relative size of errors in the primary ring versus secondary rings depends on the time discretisation choice.}
    \label{fig:test_case_2b_waveno_errors}
\end{figure}

\subsection{Gusto results}
Like in the Gaussian case, the two explicit methods have the lowest errors and RK4 is a couple of orders of magnitude more accurate than SSPRK3 (Figure \ref{fig:Test_case_2_gusto}). The implicit methods have similar errors for the smallest timesteps, but SIQN has a lower error than implicit midpoint for the largest timesteps in case 2a. This is in contrast to the Gaussian case, where SIQN was less accurate at the largest timestep. We examine this further in Figure \ref{fig:gusto_waveno_errors} where we plot $\Delta \sigma_{\vec{k}}^{\alpha}$ for the fast modes in case 2a for SIQN and implicit midpoint. At a timestep of $\Delta t =0.005 ~\text{s}$, the spectral differences are indistinguishable. However, at this largest timestep of $\Delta t = 0.05 ~\text{s}$, implicit midpoint has a large amount of error in the $(\pm5, 0) $ wavenumbers whereas SIQN has a lower total error that is spread over a wider range of wavenumbers.

\begin{figure}
     \centering
     \begin{subfigure}[t]{0.48\textwidth}
         \centering
         \includegraphics[scale=0.48]{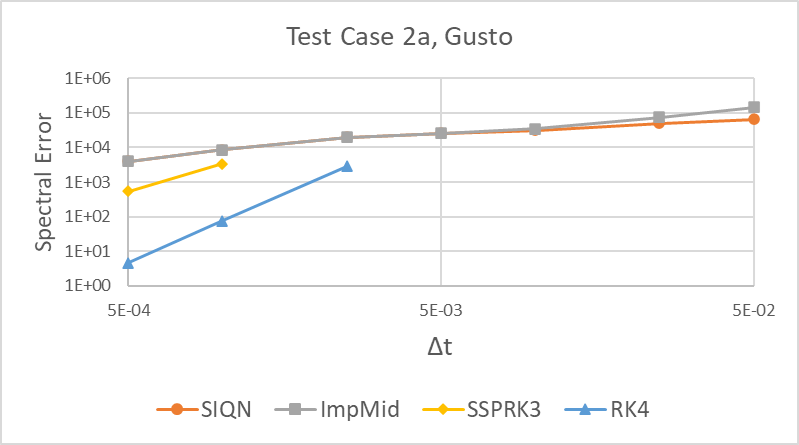}
         \caption{Test Case 2a}
     \end{subfigure}
     \hfill
     \begin{subfigure}[t]{0.48\textwidth}
         \centering
         \includegraphics[scale=0.48]{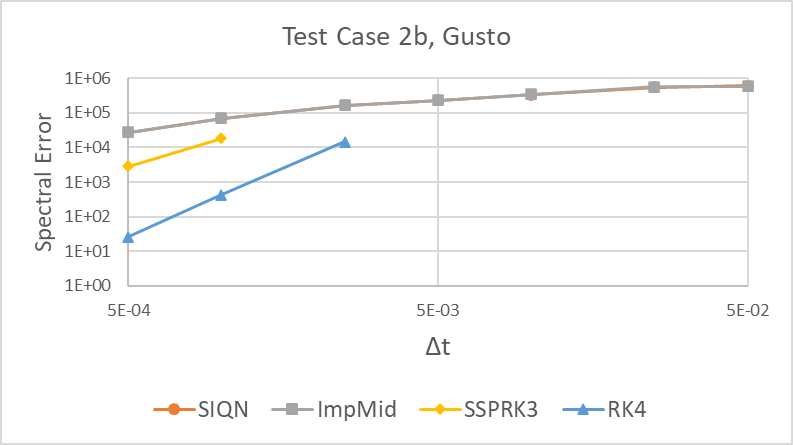}
         \caption{Test Case 2b}
     \end{subfigure}
        \caption{Spectral coefficient errors in test cases 2a and 2b for Gusto. The RK methods perform best at the smallest two timesteps in both cases (a) and (b), with RK4 having more accuracy in line with it being higher-order. SIQN has an accuracy improvement over the implicit midpoint method for the largest timesteps of test case 2a, which might be due to the higher proportion of directly resonant triadic interactions. }
    \label{fig:Test_case_2_gusto}
\end{figure}

\begin{figure}
    \centering
     \includegraphics[scale = 0.5]{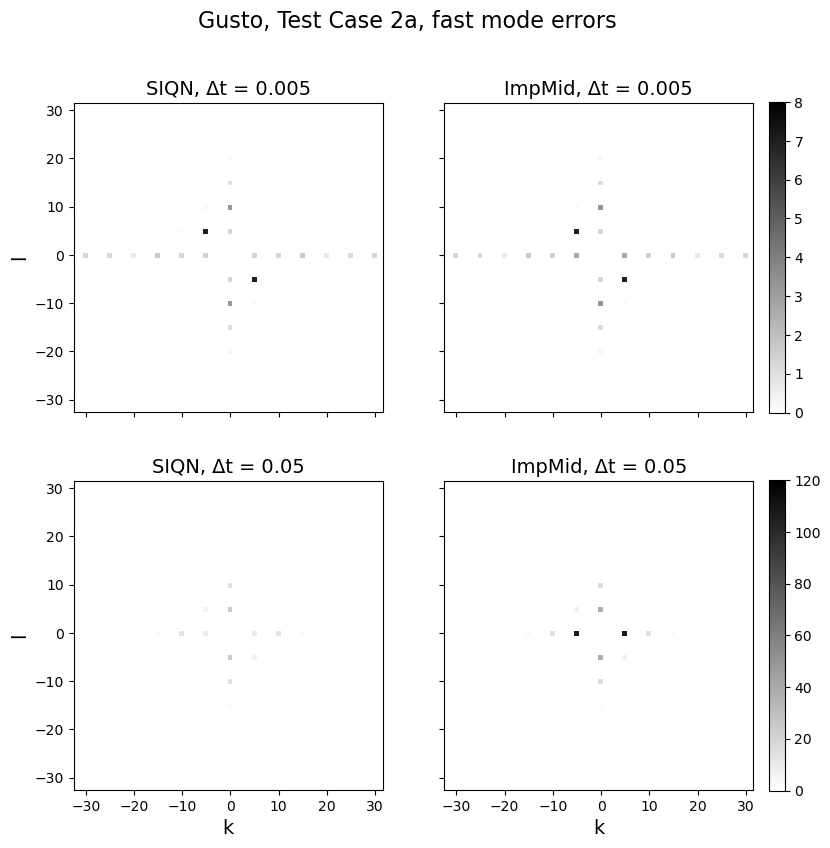}
    \caption[]{Spectral coefficient errors in the fast modes ($\Delta \sigma_{\vec{k}}^{\alpha}$) for Gusto timesteppers in test case 2a. Solutions with the two implicit methods and timesteps of $\Delta t = 0.005 ~\text{s}$ and $\Delta t = 0.05 ~\text{s}$ are compared at $T_{\text{max}}$. At the smaller timestep, the spectral coefficient errors are virtually identical. However, at the largest timestep, SIQN has the spectral errors spread over a wider range of wavenumbers, whilst implicit midpoint has a larger total error that is concentrated in the ($\pm 5, 0$) wavenumbers.}
    \label{fig:gusto_waveno_errors}
\end{figure}

\subsection{LFRic results}
The spectral errors reduce when moving from SIQN, to SSPRK3, to RK4 (Figure \ref{fig:Test_case_2_lfric}). The relative rates of error increase are similar to the Gaussian test case, with SSPRK3 having about the same error as SIQN at its largest stable timestep. RK4 again has a rapid increase in error for its largest timesteps. Interestingly, the SIQN method is stable for more timesteps in the case with fewer waves (case 2a), whereas in case 2b it cannot compute a stable solution at $\Delta t \geq 0.025 ~\text{s}$. \par

\begin{figure}
     \centering
     \begin{subfigure}[t]{0.48\textwidth}
         \centering
         \includegraphics[scale=0.48]{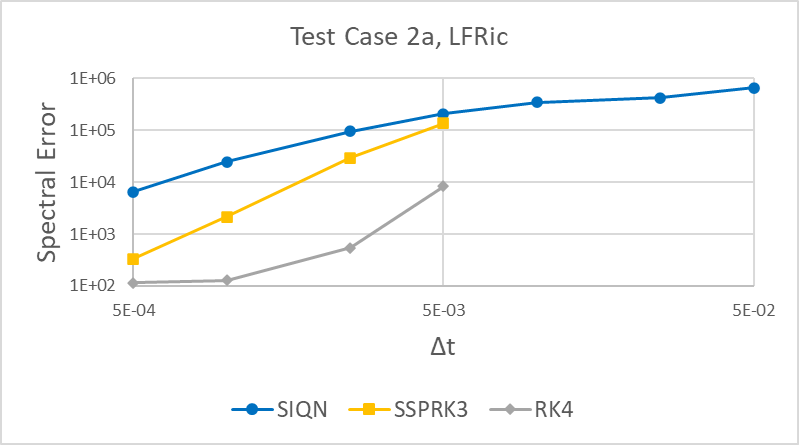}
         \caption{Test Case 2a}
     \end{subfigure}
     \hfill
     \begin{subfigure}[t]{0.48\textwidth}
         \centering
         \includegraphics[scale=0.48]{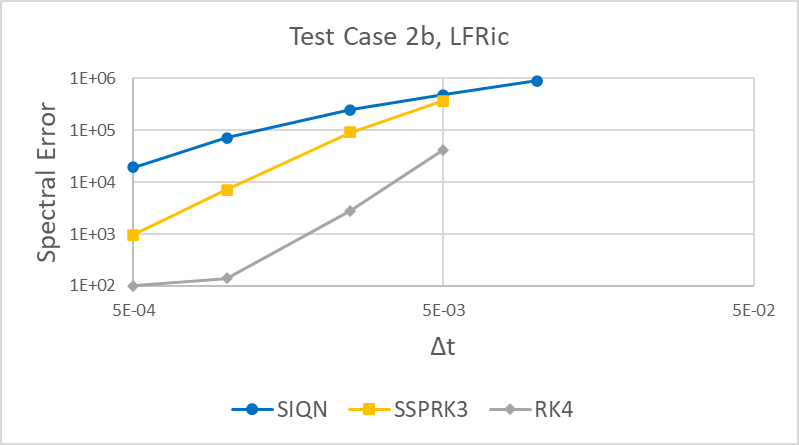}
         \caption{Test Case 2b}
     \end{subfigure}
        \caption{Spectral coefficient errors in test cases 2a and 2b for LFRic. The spectral errors are similar for both cases, except that SIQN has a stricter stability limit when there are more initial waves and nonlinear interactions in (b).}
    \label{fig:Test_case_2_lfric}
\end{figure}

In Figure \ref{fig:lfric_waveno_errors} we plot $\Delta \sigma_{\vec{k}}^{\alpha}$ at $T_{\text{max}}$ in the fast modes of case 2b for the timestep sizes of $\Delta t =0.001 ~\text{s}$ and $ \Delta t = 0.005 ~\text{s}$. At the smaller timestep, SIQN commits a higher proportion of error in the secondary ring, whilst RK4 has more error in the primary ring. When taking a larger timestep, the predominant errors for SIQN shift to the primary ring, whilst RK4's move outward to the secondary rings.

\begin{figure}
    \centering
     \includegraphics[scale = 0.5]{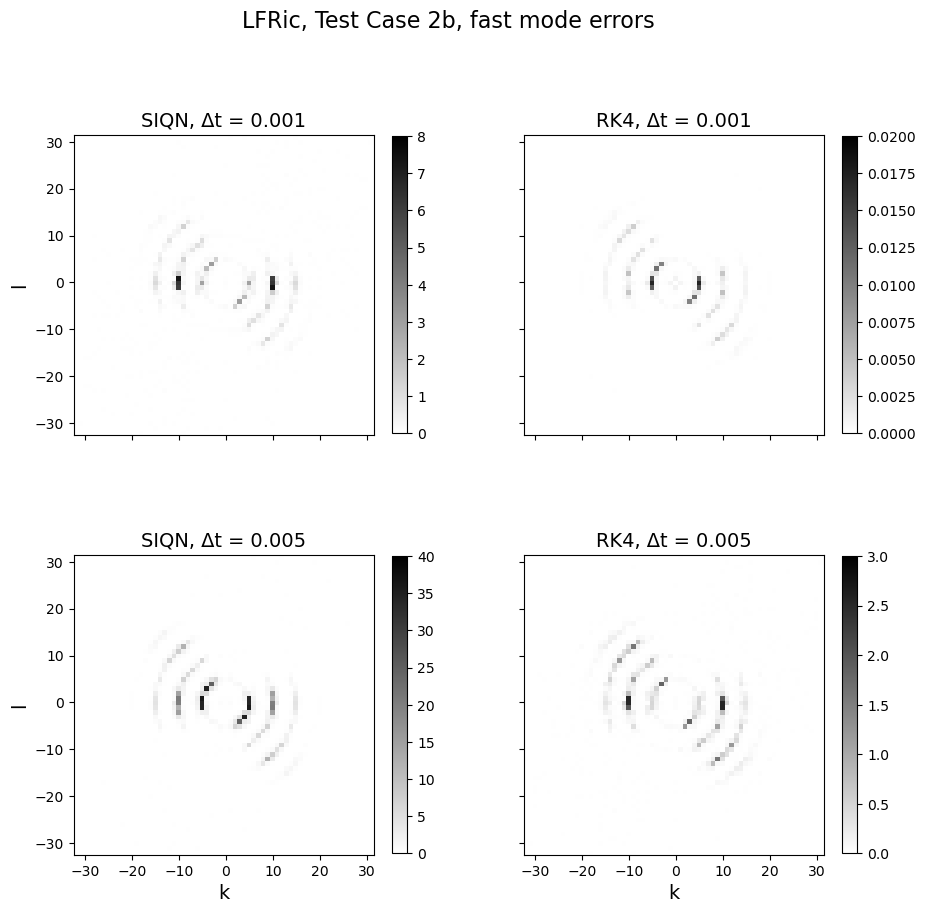}
    \caption[]{Spectral coefficient errors in the fast modes $\left( \Delta \sigma_{\vec{k}}^{\alpha=-1} + \Delta \sigma_{\vec{k}}^{\alpha=+1}  \right)$ for LFRic in case 2b. The SIQN and RK4 timesteppers are compared at $T_{\text{max}}$. With a timestep of $\Delta t =0.001$, SIQN has a larger proportion of errors in the secondary ring whilst RK4 has more in the primary ring. At a larger timestep of $\Delta t =0.005$, this trend flips, with SIQN having more error in the primary ring and RK4 in the secondary rings.}
    \label{fig:lfric_waveno_errors}
\end{figure}

\clearpage

\section{Conclusions and future work}
\label{section:conclusions}
For the next generation of time evolution models, such as for weather and climate applications, a key decision will be the choice of timestepping method and timestep size that ensures accurate solutions. In this paper, we have investigated nonlinear errors from the time discretisation when taking larger timesteps. The first part of this paper studied how linear dispersive errors impact the long time evolution of the RSWEs through inaccuracies in the frequency of the triadic interactions. Our new triadic error reflects how fast wave errors can lead to misrepresenting the slow nonlinear dynamics. \par
As the triadic frequency is not the only source of error in the spectral evolution equation (\ref{eq:sigmaeqn}), the triadic error can be extended to consider more components of the nonlinear dynamics. One idea is to consider errors in the interaction coefficients $\left( C_{\mathbf{k}_{a},\mathbf{k}_{b},\mathbf{k}}^{\alpha_{a},\alpha_{b},\alpha} \right)$ which dictate the coupling between triads. Another extension is to use a more accurate $\beta$-plane approximation for the Coriolis force, which leads to a nonzero wave speed for the slow $\alpha=0$ mode. This means that the triadic error would consider linear dispersive errors in the slow modes, and thus the slow-slow-slow interaction (type \textit{i}) becomes important to analyse. \par
The second part of this paper presented two test cases for the $f$-plane RSWEs. These initial conditions focus on a slowly developing, yet critically important, nonlinear component. The application of these initial conditions to pseudospectral and finite element (Gusto and LFRic) models shows that these cases can be used in a range of models and spatial discretisations. In the Gaussian test case, there is a low-frequency phase shift due to the nonlinearity that may be missed with large timesteps. The triadic tests also lead to slowly developing nonlinear dynamics, resulting from interactions of the initialised linear waves. The triadic cases can be made more challenging by using larger initial amplitudes. This increases the strength of the nonlinear interactions and energy exchanges, which enables the formation of more distinct rings in spectral space in case 2b. There is much potential for developing additional triadic test cases, given that there are many combinations of linear waves that could be used. \par
We are interested in using our triadic error framework and test cases to examine alternative timesteppers, especially ones that have the accuracy of explicit methods with larger stability regions like implicit methods. One such timestepper is the phase-averaged method, which first uses an exponential mapping to the modulation variable space (\ref{eq:map}) and then applies a local phase-average for larger explicit timestepping \parencite{Haut_Wingate_2014,peddle_haut_wingate_2019,hiroe_colin}. Another future direction is examining how other model choices, such as physics schemes and spatial discretisations, affect the timestepper comparisons within that model. \par

\section*{Appendix A: Eigenmodes of the $f$-plane RSWEs}
\label{appendix:eigenmodes}
To obtain a relationship between the frequencies and wavenumbers, we take the RSWEs with a skew-symmetric linear operator (\ref{eq:rswes_skew}) and represent these in Fourier space, 
\begin{equation}
    \frac{\partial \vec{\hat{U}}}{\partial t} + \hat{\mathcal{L}} \vec{\hat{U}} = \mathcal{F} \{ \mathcal{N}(\vec{U},\vec{U}) \}, 
\end{equation}
\begin{equation}
  \frac{\partial}{\partial t}  
\begin{bmatrix}
\hat{u} \\
\hat{v} \\
\hat{\phi}
\end{bmatrix}
+
\begin{bmatrix}
0 & -f & i c k\\
f & 0 & i c l\\
i c k & i c l & 0\\
\end{bmatrix}
\begin{bmatrix}
\hat{u} \\
\hat{v} \\
\hat{\phi}
\end{bmatrix}
= \mathcal{F} 
\Bigg\{ -
\begin{bmatrix}
u\frac{\partial u}{\partial x} + v \frac{\partial u}{\partial y}\\
u \frac{\partial v}{\partial x} + v  \frac{\partial v}{\partial y} \\
\frac{\partial }{\partial x} (u \phi) + \frac{\partial}{\partial y} (v \phi) 
\end{bmatrix}
\Bigg\} ,  
\end{equation}

\noindent where $c = \sqrt{g H_0}$ as the irrotational, linear, wave frequency, and $\mathcal{ F\{ \} }$ denotes the spatial Fourier transform. Solving the eigenvalue problem for the skew-Hermitian $\mathcal{\hat{L}}$ generates the dispersion relations of (\ref{eq:rswe_disp_rel}):
$$\omega_{\vec{k}}^{\alpha} = \alpha \sqrt{f^2 + c^2 K^2}, ~\alpha \in \{ -1, 0, +1 \}.$$

\noindent The slow wave mode is associated with $\alpha = 0$, and $\alpha = \pm 1$ correspond to the fast, inertia-gravity modes, which propagate in opposite directions. The eigenvectors of these modes are, for $K \neq 0$: 
\begin{equation}
    \vec{r}^{\alpha=-1}_{\vec{k}}
=
\frac{1}{\sqrt{2} K \psi}
\begin{bmatrix}
-k \psi + i f l \\
-l \psi - i f k \\
c K^2 
\end{bmatrix},
\ \ \ 
\vec{r}^{\alpha=0}_{\vec{k}}
=
\frac{1}{\psi}
\begin{bmatrix}
-i c l \\
i c k \\
f
\end{bmatrix},
\ \ \
\vec{r}^{\alpha=+1}_{\vec{k}}
=
\frac{1}{\sqrt{2} K \psi}
\begin{bmatrix}
k \psi + i f l \\
l \psi - i f k \\
c K^2 
\end{bmatrix} ,
\label{eq:eigenvecs}
\end{equation}
where, 
$$ \psi = \sqrt{f^2 + c^2 K^2}, \ \ K=\sqrt{k^2+l^2}.$$

\noindent For the case of a zero wavenumber, these become: \par
\begin{equation}
\vec{r}^{\alpha=-1}_{K=0}
=
\begin{bmatrix}
-\frac{i}{\sqrt{2}} \\
\frac{1}{\sqrt{2}}  \\
0
\end{bmatrix},
\ \ \ 
\vec{r}^{\alpha=0}_{K=0}
=
\begin{bmatrix}
0 \\
0 \\
1
\end{bmatrix},
\ \ \
\vec{r}^{\alpha=+1}_{K=0}
=
\begin{bmatrix}
\frac{i}{\sqrt{2}}  \\
\frac{1}{\sqrt{2}}  \\
0
\end{bmatrix} .
\label{eq:eigenvecs_K0}
\end{equation}

We now construct the physical space eigenmodes by taking the real part of the eigenmode projected onto the doubly periodic $f$-plane, $\vec{m}^{\alpha}_{\vec{k}} = \text{Re} \{ \exp(i(k x+l y)) ~\vec{r}_{\vec{k}}^{\alpha}   \}$:

\begin{subequations}
\begin{align}
    \vec{m}^{\alpha=-1}_{\vec{k}} 
&=
\text{Re}
\left\{e^{i(k x+l y)} \vec{r}_{\vec{k}}^{-1} 
\right\}
=
\frac{1}{\sqrt{2} K \psi}
\begin{bmatrix}
-k \psi \cos(kx+ly) - l f \sin(kx+ly) \\
-l \psi \cos(kx+ly) + k f \sin(kx+ly) \\
c K^2 \cos(kx+ly)
\end{bmatrix},
\\
\vec{m}^{\alpha=0}_{\vec{k}}
&=
\text{Re} 
\left\{
e^{i(k x+l y)} \vec{r}_{\vec{k}}^{0} 
\right\}
=
\frac{1}{\psi}
\begin{bmatrix}
l c \sin(kx+ly) \\
-k c \sin(kx+ly) \\
f \cos(kx+ly)
\end{bmatrix},
\\
\vec{m}^{\alpha=+1}_{\vec{k}}
&=
\text{Re} 
\left\{
e^{i(k x+l y)} \vec{r}_{\vec{k}}^{+1}
\right\}
=
\frac{1}{\sqrt{2} K \psi}
\begin{bmatrix}
k \psi \cos(kx+ly) - l f \sin(kx+ly) \\
l \psi \cos(kx+ly) + k f \sin(kx+ly) \\
c K^2 \cos(kx+ly)
\end{bmatrix} .
\end{align}
\end{subequations}

\section*{Appendix B: Triadic error proofs}
\label{appendix:triadic_error_proofs}
For more analysis of the order of triadic errors, we consider the fast-slow-fast (FSF/type \textit{ii}) triad. This triad is the only choice of the six interaction types in Table \ref{table:triad_list} that can construct direct- and near-resonances within a discretised domain. Letting $\omega_1$ and $\omega_3$ be the two fast wave frequencies, the triadic frequency is:
\begin{equation}
    \Omega = \omega_1 + 0 - \omega_3 = \omega_1 - \omega_3.
\end{equation}

\noindent This triad is directly resonant if $ \omega_1 = \omega_3 = \omega$, i.e. the total wavenumber of the two fast waves must be equal. If $ \omega_1 \neq \omega_3 $, then the triad is non-resonant. \par
Following (\ref{eq:num_triad_prop}), the numerical approximation to the triadic propagator of this interaction is,

\begin{equation}
    \mathcal{T}_N = P(i \omega_1 \Delta t) P(0) P(- i \omega_3 \Delta t),
\end{equation}

\noindent whilst the true triadic propagator is
\begin{equation}
    \mathcal{T}_{\Delta t} = e^{i \Omega \Delta t}.
\end{equation}

\noindent For the direct resonances,
\begin{equation}
    \mathcal{T}_{\Delta t}(\Omega = 0, \Delta t) = e^{0} = 1,
\label{eq:directly_resonant_T}
\end{equation}

\noindent and for the non-resonances,
\begin{equation}
    \mathcal{T}_{\Delta t}(\Omega \neq 0, \Delta t) = e^{i( (\omega_1 - \omega_3) \Delta t) } =  \sum_{n=0}^{\infty} \frac{(i (\omega_1 - \omega_3) \Delta t)^n}{n!}.
\label{eq:non_resonant_T}
\end{equation}

\begin{theorem}
    For non-resonant ($\Omega \neq 0$) FSF interactions, a k-order explicit Runge-Kutta method will generate a $\mathcal{O}\left( (\Delta t)^{k+1} \right)$ triadic error.
\end{theorem}

\begin{proof}

Use the Runge-Kutta linear dispersion relation (\ref{eq:P_RK}) for the individual waves in $\mathcal{T}_N$:
\begin{align}
     \mathcal{T}_N &= 
\left( \sum_{a=0}^{k} \frac{(i \omega_1 \Delta t)^a}{a!} \right) (1) \left( \sum_{b=0}^{k} \frac{(-i \omega_3 \Delta t)^b}{b!} \right), \nonumber \\
\rightarrow \mathcal{T}_N &= 
\sum_{a=0}^{k} \sum_{b=0}^{k} \frac{(i \omega_1 \Delta t)^a}{a!}   \frac{(-i \omega_3 \Delta t)^b}{b!} .
\end{align}

\noindent Now, introduce a new variable, $e=a+b$. We rewrite the summations in $a,b$ as a summation over $e$, for the combinations of $a$ and $b$ that satisfy the definition of $e$. As $a,b \in [0,k]$, $e \in [0, 2k]$:
\begin{equation}
    \mathcal{T}_N = 
\sum_{e=0}^{2k} \sum_{a+b=e} \frac{(i \omega_1 \Delta t)^a}{a!}   \frac{(-i \omega_3 \Delta t)^b}{b!}.
\end{equation}

\noindent Now, we split the sum into two parts. When $e \in [0, k]$, it must be that $a,b \in [0,e]$. When $e \in [k+1,2k]$, we have that $a,b \in [e-k,k]$. Replacing $b=e-a$ and summing over $e$ and $a$,
\begin{align}
    \mathcal{T}_N &= \sum_{e=0}^{k} \sum_{a=0}^{e} \frac{(i \omega_1 \Delta t)^a}{a!}   \frac{(-i \omega_3 \Delta t)^{e-a}}{(e-a)!} + \sum_{e=k+1}^{2k} \sum_{a=e-k}^{k} \frac{(i \omega_1 \Delta t)^a}{a!}   \frac{(-i \omega_3 \Delta t)^{e-a}}{(e-a)!}, \nonumber \\
    \rightarrow \mathcal{T}_N &= \sum_{e=0}^{k} \frac{1}{e!} \sum_{a=0}^{e} \frac{e!}{a!(e-a)!} (i \omega_1 \Delta t)^a (-i \omega_3 \Delta t)^{e-a} + A ,
\label{append_eq:triadic_proof_1_split_with_A}
\end{align}

\noindent where $A \sim \mathcal{O} \left( (\Delta t)^{k+1} \right)$ is a remainder term of
\begin{equation}
    A = \sum_{e=k+1}^{2k}  {\sum_{a=e-k}^{k}  \frac{(i \omega_1)^a}{a!}   \frac{(-i \omega_3)^{e-a}}{(e-a)!} (\Delta t)^{e} }.
\end{equation}

\noindent Now, we can simplify the first term in (\ref{append_eq:triadic_proof_1_split_with_A}) with the use of the binomial theorem, which states that  
\begin{equation}
    (x+y)^n = \sum_{a=0}^{n} \frac{n!}{a! (n-a)!} x^a y^{n-a} .
\label{append_eq:binomial}
\end{equation}

\noindent This results in the following:
\begin{align}
    \mathcal{T}_N &= \sum_{e=0}^{k} \frac{(i \omega_1 \Delta t - i \omega_3 \Delta t)^e}{e!} + A, \nonumber \\
    \rightarrow \mathcal{T}_N &= \sum_{e=0}^{k} \frac{(i (\omega_1 - \omega_3) \Delta t)^e}{e!} + A. \label{append_eq:triadic_proof_1_TN}
\end{align}

Comparing this expression with the true triadic propagator of a non-resonant FSF interaction (\ref{eq:non_resonant_T}) shows that the Runge-Kutta approximation is a truncation of the analytical expression with an $\mathcal{O}\left( (\Delta t)^{k+1} \right)$ error from the $A$ term.

\end{proof}

\begin{theorem}
   For directly resonant ($\Omega = 0$) FSF interactions, an even-order explicit Runge-Kutta method has an $\mathcal{O}\left( (\Delta t)^{k+2} \right)$ triadic error, while an odd-order method has an $\mathcal{O}\left( (\Delta t)^{k+1} \right)$ triadic error. 
\end{theorem}

\begin{proof}
Investigate the $A$ term from Theorem 1:
$$ A = \sum_{e=k+1}^{2k} \sum_{a=e-k}^{k} \frac{(i \omega_1 \Delta t)^a}{a!}   \frac{(-i \omega_3 \Delta t)^{e-a}}{(e-a)!}. $$
We extend the range of the summation in $a$ to $a \in [0,e]$ and subtract these added components:
\begin{align}
A = \sum_{e=k+1}^{2k} \Bigg( \sum_{a=0}^{e} \frac{(i \omega_1 \Delta t)^a}{a!}   \frac{(-i \omega_3 \Delta t)^{e-a}}{(e-a)!} - &\sum_{a=0}^{e-k-1} \frac{(i \omega_1 \Delta t)^a}{a!}   \frac{(-i \omega_3 \Delta t)^{e-a}}{(e-a)!} \nonumber \\
- &\sum_{a=k+1}^{e} \frac{(i \omega_1 \Delta t)^a}{a!}   \frac{(-i \omega_3 \Delta t)^{e-a}}{(e-a)!} \Bigg) .
\end{align}

\noindent We can apply the binomial theorem (\ref{append_eq:binomial}) to the first term in this sum:
\begin{align}
 A = &\sum_{e=k+1}^{2k}  \frac{(i (\omega_1 - \omega_3) \Delta t)^e}{e!} \nonumber  \\
 - &\sum_{e=k+1}^{2k} \Bigg( \sum_{a=0}^{e-k-1}\frac{(i \omega_1 \Delta t)^a}{a!}   \frac{(-i \omega_3 \Delta t)^{e-a}}{(e-a)!} - \sum_{a=k+1}^{e} \frac{(i \omega_1 \Delta t)^a}{a!}   \frac{(-i \omega_3 \Delta t)^{e-a}}{(e-a)!} \Bigg). 
\end{align}
 
\noindent We can then rewrite the numerical triadic propagator of Theorem 1 (\ref{append_eq:triadic_proof_1_TN}) as
\begin{equation}
    \mathcal{T}_N = \sum_{e=0}^{2k} \frac{(i (\omega_1 - \omega_3) \Delta t)^e}{e!} + B,
\label{append_eq:triadic_proof_2_split_with_B}
\end{equation}

\noindent where $B \sim \mathcal{O}\left( (\Delta t)^{k+1} \right)$ is
\begin{align}
    B &= - \sum_{e=k+1}^{2k} \left( \sum_{a=0}^{e-k-1} \frac{(i \omega_1 \Delta t)^a}{a!}   \frac{(-i \omega_3 \Delta t)^{e-a}}{(e-a)!} + \sum_{a=k+1}^{e} \frac{(i \omega_1 \Delta t)^a}{a!}   \frac{(-i \omega_3 \Delta t)^{e-a}}{(e-a)!} \right), \nonumber \\
     \rightarrow B &= - \sum_{e=k+1}^{2k} \left( \sum_{a=0}^{e-k-1}\frac{\omega_1^a \omega_3^{e-a}}{a!(e-a)!} (i)^e (-1)^{e-a} (\Delta t)^e + \sum_{a=k+1}^{e} \frac{\omega_1^a \omega_3^{e-a}}{a!(e-a)!} (i)^e (-1)^{e-a} (\Delta t)^e \right).
\end{align}

\noindent The first term in 
(\ref{append_eq:triadic_proof_2_split_with_B}) is a truncation of $\mathcal{T}_{\Delta t}$ up to $\mathcal{O}\left( (\Delta t)^{2k} \right)$, however, there are still lower-order errors appearing in $B$. We can consider the leading-order error in $B$, which occurs when $e = k + 1$,
\begin{align}
    B_{k+1} =  &\sum_{a=0}^{0} \frac{\omega_1^a \omega_3^{k+1-a}}{a!(k+1-a)!} (i)^{k+1} (-1)^{k+1-a} (\Delta t)^{k+1}, \nonumber \\
    + &\sum_{a=k+1}^{k+1} \frac{\omega_1^a \omega_3^{k+1-a}}{a!(k+1-a)!} (i)^{k+1} (-1)^{k+1-a} (\Delta t)^{k+1}, \nonumber \\
    \rightarrow B_{k+1} =  &- \frac{\omega_3^{k+1}}{(k+1)!} (i)^{k+1} (-1)^{k+1} (\Delta t)^{k+1} - \frac{\omega_1^{k+1} }{(k+1)!} (i)^{k+1} (\Delta t)^{k+1}, \nonumber \\
    \rightarrow B_{k+1} =  &- \frac{(i)^{k+1}}{(k+1)!} ( \omega_1^{k+1} + (-1)^{k+1} \omega_3^{k+1} ) \ (\Delta t)^{k+1}.
\end{align}

\noindent $B_{k+1}$ will reduce to zero if $( \omega_1^{k+1} + (-1)^{k+1} \omega_3^{k+1} ) = 0$. This occurs when both $\omega_1 = \omega_3$, which is the case for a direct resonance, and $(-1)^{k+1} = -1$, which is true for an even value of $k$. Thus,

$$B_{k+1} = 
\begin{cases}
0, & \text{if $k$ is even}, \\
- \frac{2 (i)^{k+1}}{(k+1)!} \omega^{k+1} (\Delta t)^{k+1}, & \text{if $k$ is odd}.
\end{cases}$$

\noindent Hence, the cancellation of $\mathcal{O}\left( (\Delta t)^{k+1} \right)$ terms for an even-order RK scheme leaves it with a higher order of truncation error than the odd-order schemes,
$$B \sim
\begin{cases}
\mathcal{O}\left( (\Delta t)^{k+2} \right), & \text{if $k$ is even}, \\
\mathcal{O}\left( (\Delta t)^{k+1} \right), & \text{if $k$ is odd}.
\end{cases}$$

\end{proof}

\begin{theorem}
A generalised Euler method generates zero triadic error for direct FSF resonances, if and only if, $\alpha = 0.5$.
\end{theorem}
\begin{proof}

Use the generalised Euler linear dispersion relation (\ref{eq:P_gen_euler}) for the individual waves in $\mathcal{T}_N$:
\begin{equation}
    \mathcal{T}_N = \left( \frac{1 + (1 - \alpha) i \omega_1 \Delta t}{1 - \alpha i \omega_1 \Delta t} \right) \left( \frac{1}{1} \right) \left( \frac{1 - (1 - \alpha) i \omega_3 \Delta t}{1 + \alpha i \omega_3 \Delta t} \right).
\end{equation}

\noindent For a directly resonant triad, $\omega_1 = \omega_3 = \omega$, and so 

\begin{equation}
    \mathcal{T}_N = \frac{1 + (1 - \alpha)^2  \omega^2 (\Delta t)^2}{1 + \alpha^2 \omega^2 (\Delta t)^2} .
\label{eq:gen_euler_proof_frac}
\end{equation}

\noindent For there to be zero triadic error, we must have that $\mathcal{T}_N = \mathcal{T}_{\Delta t} = 1$ (\ref{eq:directly_resonant_T}). This requires that $(1-\alpha)^2 = \alpha^2$ to have a matching numerator and denominator in (\ref{eq:gen_euler_proof_frac}). Thus, $1 - \alpha = \alpha \rightarrow \alpha=0.5$, which corresponds to the trapezoidal scheme.
\end{proof}

\section*{Acknowledgments}
Many thanks to James Kent and Thomas Bendall from the Met Office, for providing help with the LFRic and Gusto codes, respectively. 

\printbibliography

@book{Durran,
  title={Numerical methods for fluid dynamics: With applications to geophysics},
  author={Durran, Dale R},
  volume={32},
  year={2010},
  publisher={Springer Science \& Business Media}
}

@book{boyd2001chebyshev,
  title={Chebyshev and {F}ourier spectral methods},
  author={Boyd, John P},
  year={2001},
  publisher={Courier Corporation}
}

@book{canuto2007spectral,
  title={Spectral methods: fundamentals in single domains},
  author={Canuto, C. and Hussaini, M.Y. and Quarteroni, A. and Zang, T.A.},
  year={2007},
  publisher={Springer Science \& Business Media}
}

@article{Embid_1996,
  title={Averaging over fast gravity waves for geophysical flows with arbitrary potential vorticity},
  author={Embid, Pedro F},
  journal={Communications in Partial Differential Equations},
  volume={21},
  number={3-4},
  pages={619--658},
  year={1996},
  publisher={Taylor \& Francis}
}

@article{smith1999transfer,
  title={Transfer of energy to two-dimensional large scales in forced, rotating three-dimensional turbulence},
  author={Smith, L.M. and Waleffe, F.},
  journal={Physics of fluids},
  volume={11},
  number={6},
  pages={1608--1622},
  year={1999},
  publisher={American Institute of Physics}
}

@article{Embid_1998,
  title={Low {F}roude number limiting dynamics for stably stratified flow with small or finite {R}ossby numbers},
  author={Embid, Pedro F and Majda, Andrew J},
  journal={Geophysical \& Astrophysical Fluid Dynamics},
  volume={87},
  number={1-2},
  pages={1--50},
  year={1998},
  publisher={Taylor \& Francis}
}

@article{Chasm,
  title={Crossing the chasm: how to develop weather and climate models for next generation computers?},
  author={Lawrence, Bryan N and Rezny, Michael and Budich, Reinhard and Bauer, Peter and Behrens, J{\"o}rg and Carter, Mick and Deconinck, Willem and Ford, Rupert and Maynard, Christopher and Mullerworth, Steven and others},
  journal={Geoscientific Model Development},
  volume={11},
  number={5},
  pages={1799--1821},
  year={2018},
  publisher={Copernicus GmbH}
}

@phdthesis{Adam,
  title={Components of nonlinear oscillation and optimal averaging for stiff {PDE}s},
  author={Peddle, Adam},
  year={2018},
  school={University of Exeter}
}

@book{Majda2002,
title = "Introduction to {PDE}s and waves for the atmosphere and ocean",
author = "Majda, A.",
year = "2002",
language = "English (US)",
series = "Courant Lecture Notes",
publisher = "New York University, Courant Institute of Mathematical Sciences and American Mathematical Society",

}

@article{Schochet,
  title={Fast singular limits of hyperbolic {PDE}s},
  author={Schochet, Steven},
  journal={Journal of differential equations},
  volume={114},
  number={2},
  pages={476--512},
  year={1994},
  publisher={Elsevier}
}

@article{green_comp,
  title={Green computing: {C}urrent research trends},
  author={Saha, Biswajit},
  journal={International Journal of Computer Sciences and Engineering},
  volume={6},
  number={3},
  pages={467--469},
  year={2018}
}

@article{cox_mathews,
  title={Exponential time differencing for stiff systems},
  author={Cox, Steven M and Matthews, Paul C},
  journal={Journal of Computational Physics},
  volume={176},
  number={2},
  pages={430--455},
  year={2002},
  publisher={Elsevier}
}

@article{TR-BDF2,
  title={Analysis and implementation of {TR-BDF}2},
  author={Hosea, ME and Shampine, LF},
  journal={Applied Numerical Mathematics},
  volume={20},
  number={1-2},
  pages={21--37},
  year={1996},
  publisher={Elsevier}
}

@phdthesis{Alex_Owen_thesis,
  title={Resonant effects in weakly nonlinear geophysical fluid dynamics},
  author={Owen, Alexander},
  year={2019},
  school={University of Exeter}
}

@article{Ward_Dewar,
  title={Scattering of gravity waves by potential vorticity in a shallow-water fluid},
  author={Ward, Marshall L and Dewar, William K},
  journal={Journal of Fluid Mechanics},
  volume={663},
  pages={478--506},
  year={2010},
  publisher={Cambridge University Press}
}

@article{Smith_Lee,
  title={On near resonances and symmetry breaking in forced rotating flows at moderate {R}ossby number},
  author={Smith, Leslie M and Lee, Youngsuk},
  journal={Journal of Fluid Mechanics},
  volume={535},
  pages={111--142},
  year={2005},
  publisher={Cambridge University Press}
}

@book{iserlies,
  title={A first course in the numerical analysis of differential equations},
  author={Iserles, Arieh},
  number={44},
  year={2009},
  publisher={Cambridge university press}
}

@book{SandersVerhulst,
title = "Averaging methods in nonlinear dynamical systems",
author = "Sanders, J. A. and Verhulst, F. and Murdock, J.",
edition = "2",
series = "Applied mathematical sciences",
publisher = "Springer",
address = "New York, Berlin, Heidelberg",
year = 2007
}

@article{James_Kent_LFRic,
  title={A mixed finite-element discretisation of the shallow-water equations},
  author={Kent, James and Melvin, Thomas and Wimmer, Golo Albert},
  journal={Geoscientific Model Development},
  volume={16},
  number={4},
  pages={1265--1276},
  year={2023},
  publisher={Copernicus GmbH}
}

@article{Galewsky_test,
  title={An initial-value problem for testing numerical models of the global shallow-water equations},
  author={Galewsky, Joseph and Scott, Richard K and Polvani, Lorenzo M},
  journal={Tellus A: Dynamic Meteorology and Oceanography},
  volume={56},
  number={5},
  pages={429--440},
  year={2004},
  publisher={Taylor \& Francis}
}

@article{nineteen_dubious_ways,
  title={Nineteen dubious ways to compute the exponential of a matrix},
  author={Moler, Cleve and Van Loan, Charles},
  journal={SIAM review},
  volume={20},
  number={4},
  pages={801--836},
  year={1978},
  publisher={SIAM}
}

@article{Haut_Babb,
  title={A high-order time-parallel scheme for solving wave propagation problems via the direct construction of an approximate time-evolution operator},
  author={Haut, Terry S and Babb, T and Martinsson, PG and Wingate, BA},
  journal={IMA Journal of Numerical Analysis},
  volume={36},
  number={2},
  pages={688--716},
  year={2016},
  publisher={Oxford University Press}
}

@article{Aechtner_wavelet,
  title={A conservative adaptive wavelet method for the shallow-water equations on the sphere},
  author={Aechtner, Matthias and Kevlahan, NK-R and Dubos, Thomas},
  journal={Quarterly Journal of the Royal Meteorological Society},
  volume={141},
  number={690},
  pages={1712--1726},
  year={2015},
  publisher={Wiley Online Library}
}

@article{Ullrich_Jablonowski,
  title={High-order finite-volume methods for the shallow-water equations on the sphere},
  author={Ullrich, Paul A and Jablonowski, Christiane and Van Leer, Bram},
  journal={Journal of Computational Physics},
  volume={229},
  number={17},
  pages={6104--6134},
  year={2010},
  publisher={Elsevier}
}

@article{Ullrich_Lauritzen,
  title={A high-order fully explicit flux-form semi-{L}agrangian shallow-water model},
  author={Ullrich, Paul Aaron and Lauritzen, Peter Hjort and Jablonowski, Christiane},
  journal={International Journal for Numerical Methods in Fluids},
  volume={75},
  number={2},
  pages={103--133},
  year={2014},
  publisher={Wiley Online Library}
}

@article{Paldor_quantiative_test_case,
  title={A quantitative test case for global-scale dynamical cores based on analytic wave solutions of the shallow-water equations},
  author={Shamir, Ofer and Paldor, Nathan},
  journal={Quarterly Journal of the Royal Meteorological Society},
  volume={142},
  number={700},
  pages={2705--2714},
  year={2016},
  publisher={Wiley Online Library}
}

@article{Paldor_Matsuno,
  title={The {M}atsuno baroclinic wave test case},
  author={Shamir, Ofer and Yacoby, Itamar and Ziskin Ziv, Shlomi and Paldor, Nathan},
  journal={Geoscientific Model Development},
  volume={12},
  number={6},
  pages={2181--2193},
  year={2019},
  publisher={Copernicus Publications G{\"o}ttingen, Germany}
}

@article{Williamson_tests,
  title={A standard test set for numerical approximations to the shallow water equations in spherical geometry},
  author={Williamson, David L and Drake, John B and Hack, James J and Jakob, R{\"u}diger and Swarztrauber, Paul N},
  journal={Journal of computational physics},
  volume={102},
  number={1},
  pages={211--224},
  year={1992},
  publisher={Elsevier}
}

@article{Staniforth_hor_grids,
  title={Horizontal grids for global weather and climate prediction models: a review},
  author={Staniforth, Andrew and Thuburn, John},
  journal={Quarterly Journal of the Royal Meteorological Society},
  volume={138},
  number={662},
  pages={1--26},
  year={2012},
  publisher={Wiley Online Library}
}

@article { Weller_comp_modes,
      author = "Hilary Weller and John Thuburn and Colin J. Cotter",
      title = "Computational Modes and Grid Imprinting on Five Quasi-Uniform Spherical {C} Grids",
      journal = "Monthly Weather Review",
      year = "2012",
      publisher = "American Meteorological Society",
      address = "Boston MA, USA",
      volume = "140",
      number = "8",
      doi = "https://doi.org/10.1175/MWR-D-11-00193.1",
      pages=      "2734 - 2755"
}

@article{Mengaldo_time_int,
  title={Current and emerging time-integration strategies in global numerical weather and climate prediction},
  author={Mengaldo, Gianmarco and Wyszogrodzki, Andrzej and Diamantakis, Michail and Lock, Sarah-Jane and Giraldo, Francis X and Wedi, Nils P},
  journal={Archives of Computational Methods in Engineering},
  volume={26},
  pages={663--684},
  year={2019},
  publisher={Springer}
}

@article{hiroe_colin,
  title={Time parallel integration and phase averaging for the nonlinear shallow water equations on the sphere},
  author={Yamazaki, Hiroe and Cotter, Colin J and Wingate, Beth A},
  journal={Quarterly Journal of the Royal Meteorological Society},
  year={2023},
  publisher={Wiley/Royal Meteorological Society}
}

@article{Newell_rossby,
  title={Rossby wave packet interactions},
  author={Newell, Alan C},
  journal={Journal of Fluid Mechanics},
  volume={35},
  number={2},
  pages={255--271},
  year={1969},
  publisher={Cambridge University Press}
}

@article{luan_exp_int_large_dt,
  title={Further development of efficient and accurate time integration schemes for meteorological models},
  author={Luan, Vu Thai and Pudykiewicz, Janusz A and Reynolds, Daniel R},
  journal={Journal of Computational Physics},
  volume={376},
  pages={817--837},
  year={2019},
  publisher={Elsevier}
}

@article{cotter_shipton_mixed_FE,
  title={Mixed finite elements for numerical weather prediction},
  author={Cotter, Colin J and Shipton, Jemma},
  journal={Journal of Computational Physics},
  volume={231},
  number={21},
  pages={7076--7091},
  year={2012},
  publisher={Elsevier}
}

@article{Melvin_gungho,
  title={A mixed finite-element, finite-volume, semi-implicit discretization for atmospheric dynamics: Cartesian geometry},
  author={Melvin, Thomas and Benacchio, Tommaso and Shipway, Ben and Wood, Nigel and Thuburn, John and Cotter, Colin},
  journal={Quarterly Journal of the Royal Meteorological Society},
  volume={145},
  number={724},
  pages={2835--2853},
  year={2019},
  publisher={Wiley Online Library}
}

@article{dynamical_core_comparison,
  title={{DCMIP}2016: a review of non-hydrostatic dynamical core design and intercomparison of participating models},
  author={Ullrich, Paul A and Jablonowski, Christiane and Kent, James and Lauritzen, Peter H and Nair, Ramachandran and Reed, Kevin A and Zarzycki, Colin M and Hall, David M and Dazlich, Don and Heikes, Ross and others},
  journal={Geoscientific Model Development},
  volume={10},
  number={12},
  pages={4477--4509},
  year={2017},
  publisher={Copernicus GmbH}
}

@article{Haut_Wingate_2014,
  title={An asymptotic parallel-in-time method for highly oscillatory {PDE}s},
  author={Haut, Terry and Wingate, Beth},
  journal={SIAM Journal on Scientific Computing},
  volume={36},
  number={2},
  pages={A693--A713},
  year={2014},
  publisher={SIAM}
}

@article{Zerrokaut_allen,
  title={A moist {B}oussinesq shallow water equations set for testing atmospheric models},
  author={Zerroukat, Mohamed and Allen, Thomas},
  journal={Journal of Computational Physics},
  volume={290},
  pages={55--72},
  year={2015},
  publisher={Elsevier}
}

@article{Pope_exp_ints,
  title={An exponential method of numerical integration of ordinary differential equations},
  author={Pope, David A},
  journal={Communications of the ACM},
  volume={6},
  number={8},
  pages={491--493},
  year={1963},
  publisher={ACM New York, NY, USA}
}

@article{Hochburck_Ostermann_exp_ints,
  title={Exponential integrators},
  author={Hochbruck, Marlis and Ostermann, Alexander},
  journal={Acta Numerica},
  volume={19},
  pages={209--286},
  year={2010},
  publisher={Cambridge University Press}
}

@phdthesis{TR_BDF2_analysis,
  title={An analysis of the TR-BDF2 integration scheme},
  author={Dharmaraja, Sohan},
  year={2007},
  school={Massachusetts Institute of Technology}
}

@article{melvin2024mixed,
  title={A mixed finite-element, finite-volume, semi-implicit discretisation for atmospheric dynamics: Spherical geometry},
  author={Melvin, Thomas and Shipway, Ben and Wood, Nigel and Benacchio, Tommaso and Bendall, Thomas and Boutle, Ian and Brown, Alex and Johnson, Christine and Kent, James and Pring, Stephen and others},
  journal={arXiv preprint arXiv:2402.13738},
  year={2024}
}

@article{raupp2009resonant,
  title={Resonant wave interactions in the presence of a diurnally varying heat source},
  author={Raupp, Carlos FM and Dias, Pedro L Silva},
  journal={Journal of the Atmospheric Sciences},
  volume={66},
  number={10},
  pages={3165--3183},
  year={2009},
  publisher={American Meteorological Society}
}

@article{kartashova2007model,
  title={Model of intraseasonal oscillations in {E}arth’s atmosphere},
  author={Kartashova, Elena and L’vov, Victor S},
  journal={Physical review letters},
  volume={98},
  number={19},
  pages={198501},
  year={2007},
  publisher={APS}
}

@Article{raphaldini2022precession,
  title={Precession resonance of {R}ossby wave triads and the generation of low-frequency atmospheric oscillations},
  author={Raphaldini, Breno and Peixoto, Pedro da Silva and Teruya, Andr{\'e} Seiji Watake and Raupp, Carlos Frederico Mendon{\c{c}}a and Bustamante, Miguel D},
  journal={Physics of Fluids},
  volume={34},
  number={7},
  year={2022},
  publisher={AIP Publishing}
}

@article{craik1971non,
  title={Non-linear resonant instability in boundary layers},
  author={Craik, Alex DD},
  journal={Journal of Fluid Mechanics},
  volume={50},
  number={2},
  pages={393--413},
  year={1971},
  publisher={Cambridge University Press}
}

@article{bretherton1964resonant,
  title={Resonant interactions between waves: {T}he case of discrete oscillations},
  author={Bretherton, Francis P},
  journal={Journal of Fluid Mechanics},
  volume={20},
  number={3},
  pages={457--479},
  year={1964},
  publisher={Cambridge University Press}
}

@article{mcewan1971degeneration,
  title={Degeneration of resonantly-excited standing internal gravity waves},
  author={McEwan, AD},
  journal={Journal of Fluid Mechanics},
  volume={50},
  number={3},
  pages={431--448},
  year={1971},
  publisher={Cambridge University Press}
}

@article{bourget2013experimental,
  title={Experimental study of parametric subharmonic instability for internal plane waves},
  author={Bourget, Baptiste and Dauxois, Thierry and Joubaud, Sylvain and Odier, Philippe},
  journal={Journal of Fluid Mechanics},
  volume={723},
  pages={1--20},
  year={2013},
  publisher={Cambridge University Press}
}

@article{lannelongue2021green,
  title={Green algorithms: quantifying the carbon footprint of computation},
  author={Lannelongue, Lo{\"\i}c and Grealey, Jason and Inouye, Michael},
  journal={Advanced science},
  volume={8},
  number={12},
  pages={2100707},
  year={2021},
  publisher={Wiley Online Library}
}

@article{dahlquist1963special,
  title={A special stability problem for linear multistep methods},
  author={Dahlquist, Germund G},
  journal={BIT Numerical Mathematics},
  volume={3},
  number={1},
  pages={27--43},
  year={1963},
  publisher={Springer}
}

@article{trefethen1982group,
  title={Group velocity in finite difference schemes},
  author={Trefethen, Lloyd N},
  journal={SIAM review},
  volume={24},
  number={2},
  pages={113--136},
  year={1982},
  publisher={SIAM}
}

@article{wingate_timestep_alpha,
  title={The maximum allowable time step for the shallow water $\alpha$ model and its relation to time-implicit differencing},
  author={Wingate, BA},
  journal={Monthly weather review},
  volume={132},
  number={12},
  pages={2719--2731},
  year={2004},
  publisher={American Meteorological Society}
}

@article{long2011numerical,
  title={Numerical wave propagation on non-uniform one-dimensional staggered grids},
  author={Long, David and Thuburn, John},
  journal={Journal of Computational Physics},
  volume={230},
  number={7},
  pages={2643--2659},
  year={2011},
  publisher={Elsevier}
}

@Article{peddle_haut_wingate_2019,
  title={Parareal convergence for oscillatory {PDE}s with finite time-scale separation},
  author={Peddle, Adam G and Haut, Terry and Wingate, Beth},
  journal={SIAM Journal on Scientific Computing},
  volume={41},
  number={6},
  pages={A3476--A3497},
  year={2019},
  publisher={SIAM}
}

@article{hartney2024compatible,
  title={A compatible finite element discretisation for moist shallow water equations},
  author={Hartney, Nell and Bendall, Thomas M and Shipton, Jemma},
  journal={arXiv preprint arXiv:2409.07182},
  year={2024}
}

@article{fu1981observations,
  title={Observations and models of inertial waves in the deep ocean},
  author={Fu, Lee-Lueng},
  journal={Reviews of Geophysics},
  volume={19},
  number={1},
  pages={141--170},
  year={1981},
  publisher={Wiley Online Library}
}

@article{chen2005resonant,
  title={Resonant interactions in rotating homogeneous three-dimensional turbulence},
  author={Chen, Qiaoning and Chen, Shiyi and Eyink, Gregory L and Holm, Darryl D},
  journal={Journal of Fluid Mechanics},
  volume={542},
  pages={139--164},
  year={2005},
  publisher={Cambridge University Press}
}

\end{document}